\documentclass{amsart}

\usepackage{amssymb,amsmath,amsthm,amsfonts}
\usepackage{setspace}
\usepackage{graphicx}
\usepackage{mathabx}
\usepackage[usenames,dvipsnames,svgnames,table]{xcolor}
\usepackage{enumerate}
\usepackage{comment}
\usepackage{latexsym}
\usepackage{hyperref}
\hypersetup{
  colorlinks   = true,    
  urlcolor     = blue,    
  linkcolor    = blue,    
  citecolor    = purple      
}

\newtheorem{thm}{Theorem}[section]
\newtheorem{prop}[thm]{Proposition}
\newtheorem{lem}[thm]{Lemma}

\newtheorem{definition}[thm]{Definition}

\newcommand{\R}{\mathbb{R}}

\newcommand{\Z}{\mathbb{Z}}

\newcommand{\lb}{\langle}
\newcommand{\rb}{\rangle}
\renewcommand{\d}{\,\operatorname{d}\!}
\renewcommand{\S}{\operatorname{S}}

\allowdisplaybreaks

\begin{document}

\title[Smoothing for the Zakharov \& KGS Systems]{Smoothing for the Zakharov \& Klein-Gordon-Schr\"{o}dinger Systems on Euclidean Spaces}

\author{E. Compaan}
\address{Department of Mathematics, University of Illinois, Urbana, IL}
\thanks{The author was supported by a National Physical Science Consortium fellowship.}
\email{compaan2@illinois.edu}
\urladdr{www.math.uiuc.edu/~compaan2} 

\begin{abstract}This paper studies the regularity of solutions to the Zakharov and Klein-Gordon-Schr\"{o}dinger systems at low regularity levels. The main result is that the nonlinear part of the solution flow falls in a smoother space than the initial data. This relies on a new bilinear $X^{s,b}$ estimate, which is proved using delicate dyadic and angular decompositions of the frequency domain. Such smoothing estimates have a number of implications for the long-term dynamics of the system. In this work, we give a simplified proof of the existence of global attractors for the Klein-Gordon-Schr\"{o}dinger flow in the energy space for dimensions $d = 2,3$. Secondly, we use smoothing in conjunction with a high-low decomposition to show global well-posedness of the Klein-Gordon-Schr\"{o}dinger evolution on $\mathbb{R}^4$ below the energy space for sufficiently small initial data. 
\end{abstract}

\maketitle


\section{Introduction}

In this work, we derive smoothing estimates for the Klein-Gordon Schr\"{o}dinger system with Yukawa coupling: 
\begin{equation} \label{eq:KGS}
\begin{cases}
i u _t + \Delta u = - uv, \quad x \in \R^d,\; t \in \R\\ 
v_{tt} + (-\Delta + 1) v = | u |^2 \\
\Big(u(\cdot, 0), v(\cdot, 0), v_t(\cdot, 0)\Big) = (u_0, v_0, v_1)   \in H^s \times H^r \times H^{r-1}.
\end{cases}
\end{equation}
We also consider the related Zakharov system:
\begin{equation} \label{eq:Z}
\begin{cases}
i u _t + \Delta u = un, \quad x \in \R^d,\; t \in \R\\ 
n_{tt} -\Delta n = \Delta | u |^2 \\
\Big(u(\cdot, 0 ), n(\cdot, 0), n_t(\cdot,0)\Big) = (u_0, n_0, n_1)\in H^s \times H^r \times H^{r-1}.
\end{cases}
\end{equation}
The Klein-Gordon-Schr\"{o}dinger system \eqref{eq:KGS} is a model from classical particle physics, in which $u$ represents a complex nucleon field and $v$ a real meson field \cite{FT}. The Zakharov system \eqref{eq:Z} was introduced in \cite{Zak} to model Langmuir turbulence in ionized plasma. In it, the function $u$ represents the envelope of a oscillating electric field while $n$ represents the deviation of ion density from its average value. 

Solutions of the Klein-Gordon-Schr\"{o}dinger system conserve the mass and the Hamiltonian energy:
\[M(u)= \| u \|_{L^2} \]
\[ E(u, v, v_t) =\|\nabla u \|_{L^2}^2  + \frac12 \Big(\| v\|_{L^2}^2 + \|v_t\|_{L^2}^2 + \|\nabla v\|_{L^2}^2 \Big) - \int|u|^2v \d x.\]
Note that the energy space for the Klein-Gordon-Schr\"{o}dinger system is $H^1 \times H^1 \times L^2$. Similarly, the Zakharov system has the following conservation laws:
\[M(u)= \| u \|_{L^2} \]
\[\widetilde{E}(u, n, n_t) = \|\nabla u \|_{L^2}^2+ \frac12\Big(\| n\|_{L^2}^2 + \|(-\Delta)^{-1/2}n_t\|_{L^2}^2 \Big)  + \int|u|^2n \d x. \]
This law identifies the energy space as $H^1 \times L^2 \times \dot{H}^{-1}$. 

The wellposedness theory for the Zakharov system on Euclidean spaces has been extensively studied. In \cite{SS} existence results for smooth solutions in dimensions $d \leq 3$ were derived. The regularity assumptions and dimension restrictions were weakened in \cite{AA, OT, KPV, SW}. In \cite{GTV}, Ginibre, Tsutsumi, and Velo applied Bourgain's restricted norm method \cite{B1} to obtain local existence results in all dimensions, covering the full subcritical regularity range (excluding the endpoints) for $d\geq 4$. In dimension $d=1$, they obtained local existence at the critical regularity $L^2 \times H^{-\frac12} \times H^{-\frac32}$. In \cite{Hol}, local ill-posedness results were obtained for some regularities outside the well-posedness regime  established in \cite{GTV}. In dimensions two and three, the local well-posedness was obtained in the critical space $L^2 \times H^{-\frac12} \times H^{-\frac32}$ in \cite{BHHT} and \cite{BH} respectively. These results are sharp in the sense that the data-to-solution map fails to be analytic at lower regularity levels. 

In one dimension, the Hamiltonian conservation law upgrades local existence to global for initial data in $H^1 \times L^2 \times (-\Delta)^{\frac12} L^2$. This result was improved in \cite{P3, P4} using Bourgain's high-low decomposition \cite{B2} method and the I-method \cite{CKSTT} respectively. It was lowered further to global existence in $L^2 \times H^{-1/2} \times H^{-3/2}$ in \cite{CHT} using an iteration method relying on the $L^2$ conservation of $u$. In two and three dimensions, global existence in the energy space follows from the Hamiltonian conservation as long as $\|u_0\|_{L^2}$ is sufficiently small. In two dimensions, global well-posedness for some regularities below the energy space was obtained in \cite{Kis} using the I-method.

Unlike the Zakharov, the nonlinearity in the wave part of the Klein-Gordon-Schr\"{o}dinger system contains no derivative. Thus we have well-posedness at somewhat lower regularity levels for this system. For the two-dimensional Klein-Gordon-Schr\"{o}dinger, local well-posedness holds in $ H^{-\frac14 +} \times H^{-\frac12}\times H^{-\frac32}$; see \cite{P1}. The same result, up to endpoints, hold in three dimensions \cite{P2}. Local existence in higher dimensions follows from the estimates derived for the Zakharov in \cite{GTV}. 

For the Klein-Gordon-Schr\"{o}dinger system in dimensions $d \leq 3$, global existence in $H^1 \times H^1 \times L^2$ follows from the Hamiltonian conservation law. In three dimensions, global existence somewhat below the energy was proved using Bourgain's high-low decomposition method in \cite{P5}. This was improved in \cite{Tz}, where the I-method was used to obtain global existence below the energy for $d \leq 3$. Global existence for the three-dimensional KGS in $L^2 \times L^2 \times H^{-1}$ was obtained in \cite{CHT}, again relying on the $L^2$ conservation law for $u$. This was lowered to $L^2 \times H^{-\frac12}\times H^{-\frac32}$ for $d=2$ and to $L^2 \times H^{-\frac12 +} \times H^{-\frac32 +}$ for $d=3$ in \cite{P1}. We also note that global existence for the closely related wave-Schr\"{o}dinger system on $H^s \times H^r \times H^{r-1}$ for some $s,r < 0$ was shown in \cite{Ak}. 

This paper is concerned with the dynamics of solutions to \eqref{eq:KGS} and \eqref{eq:Z}. The main result is that the difference between the linear evolution and the nonlinear evolution resides in a higher-regularity space. This follows from a refinement of the bilinear $X^{s,b}$ local theory estimates, similar to that contained in \cite{CDKS} for two dimensional nonlinear Schr\"{o}dinger equations with quadratic nonlinearities. The difficulty in this case is that the addition of a Klein-Gordon or wave equation to the Schr\"{o}dinger to obtain \eqref{eq:KGS} or \eqref{eq:Z} respectively complicates the resonance relations in the system, making the estimates more challenging. As in \cite{CDKS}, the proof depends on delicate decompositions of the frequency space to control the nonlinear interactions, with especial care being required near the resonant sets of the interaction. 

In the remainder of the paper we present some consequences of the smoothing estimate. One of these is a simplified proof of the existence of a global attractor for the forced and damped Klein-Gordon-Schr\"{o}dinger equation in dimensions $d =2,3$. This result is known \cite{LW}, but the existing proof relies on truncation arguments to obtain the necessary compactness. The truncation step can be eliminated by the employment of the smoothing effect of the nonlinear flow, significantly simplifying the argument. Secondly, we show global existence below the energy space for the four-dimensional Klein-Gordon-Schr\"{o}dinger system for $\|u_0\|_{L^2}$ sufficiently small using a variant of Bourgain's high-low argument \cite{B2}. Similar smoothing estimates have been used with high-low decomposition method to prove global existence for other equations -- see e.g. \cite{DET} for results on the periodic fractional Schr\"{o}dinger equation. We remark that method of \cite{CHT} to obtain global existence for the Klein-Gordon-Schr\"{o}dinger does not apply; in four dimensions, there is not sufficient slack in the wave equation estimates to iterate that scheme. The refinement used in \cite{P1}, which uses $X^{s,b}$ estimates instead of Strichartz space controls, also cannot be directly applied. Smoothing estimates provide a straightforward proof of the global existence.

The paper is organized as follows. In Section \ref{notation}, we introduce the function spaces required for the estimates, and in Section \ref{statements}, we state the results of the paper. Sections \ref{smoothingProof}, \ref{gaProof}, and \ref{R4Proof} contain the proofs of the main smoothing estimate, the existence of the Klein-Gordon-Schr\"{o}dinger attractor, and global well-posedness in $\mathbb{R}^4$ respectively. Finally, in Section \ref{estimateProof}, we prove the main bilinear estimate. 

\section{Notation \& Function Spaces} \label{notation}

Before stating the results of this paper, we need some definitions. The Fourier sequence of a function $u \in L^2(\mathbb{R}^d)$ is defined by 
\[ \hat{u}(\xi) = \int _{\mathbb{R}^d}  u(x) e^{-i\xi\cdot x} \; \d x. \]
We use Sobolev spaces $H^s(\mathbb{R}^d)$, with their norms given by 
\[ \| u \|_{H^s} = \| \lb \xi \rb^s \hat{u}(\xi) \|_{L^2}, \]
where $\lb \xi \rb = (1 + |\xi|^2)^{1/2}$. We also use the homogeneous space $\dot{H}^s$, with $\|u\|_{\dot{H}^s} = \| |\xi|^s \hat{u}(\xi)\|_{L^2}$.  

To prove the desired estimates, we work with transformed versions of the systems \eqref{eq:KGS} and \eqref{eq:Z}. Define $A = (1 - \Delta)^{1/2}$. For the Klein-Gordon-Schr\"{o}dinger system, let $v^\pm = v \pm i A^{-1} v_t$. Under this transformation, \eqref{eq:KGS} becomes
\begin{equation} \label{eq:tKGS}
\begin{cases}
i u _t + \Delta u = -\frac12  u(v^+ + v^-), \\
iv_t^\pm  \mp Av^\pm  = \mp A^{-1}|u|^2 \\
\Big(u(\cdot, 0), v^\pm(\cdot, 0)\Big) = (u_0, v_0^\pm)   \in H^s \times H^r.
\end{cases}
\end{equation}

For the Zakharov system \eqref{eq:Z}, we similarly define $n^\pm = n \pm i A^{-1} n_t$. After this transformation, \eqref{eq:Z} becomes 
\begin{equation} \label{eq:tZ}
\begin{cases}
i u _t + \Delta u = \frac12 u(n^+ + n^-), \\ 
in_t^\pm \mp A n^\pm  = \mp A^{-1}\Delta | u |^2 \mp A^{-1} \operatorname{Re} n^\pm \\
\Big(u(\cdot, 0 ), n^\pm(\cdot, 0), n_t(\cdot,0)\Big) = (u_0, n_0^\pm )\in H^s \times H^r.
\end{cases}
\end{equation}
Notice that we can recover the original function $v$ and $n$ by taking the real part of $v^\pm$ and $n^\pm$ respectively. The corresponding Bourgain spaces are defined by the norms
\begin{align*}
 \|u\|_{X^{s,b}} &= \|\lb \xi \rb^{s} \lb \tau + |\xi|^2\rb^{b} \hat{u}(\xi, \tau)\|_{L^2_{\xi,\tau}} \\
 \|v\|_{X^{s,b}_\pm} &= \| \lb \xi \rb^{s} \lb \tau \pm |\xi | \rb^{b} \hat{u}(\xi,\tau)\|_{L^2_{\xi,\tau}}. 
\end{align*}
The multiplier for the Klein-Gordon evolution is technically $\lb \tau \pm \lb \xi\rb\rb$ rather than $\lb \tau \pm | \xi| \rb$, but $\lb \tau\pm \lb \xi \rb \rb \approx \lb \tau \pm |\xi| \rb$ and using the latter multiplier results in a cleaner exposition. We also define the time-restricted versions of these norms:
\begin{align*}
 \|u\|_{X^{s,b}_\delta} &= \inf_{u = \tilde{u}, |t | \leq \delta}\| \tilde{u}\|_{X^{s,b}} \\
 \|v\|_{X^{s,b}_{\pm, \delta}} &= \inf_{v = \tilde{v}, |t| \leq \delta} \| \tilde{v}\|_{X^{s,b}_\pm}. 
\end{align*}

The expression $e^{-tL}u_0$ will denote the solution to the linear problem $u_t  + Lu = 0$ with $u(\cdot, 0) = u_0$. Thus, for example, $e^{it\Delta}u_0$ is the linear Schr\"{o}dinger flow. 
We will write $a \lesssim b$ to indicate that there is an absolute constant $C$ such that $a \leq Cb$. The symbols $\gtrsim$ is used similarly. The expression $a \approx b$ means that $a \lesssim b$ and $a \gtrsim b$. We write $a - $ for $a - \epsilon$ when $\epsilon > 0$ is arbitrary; similarly we write $a +$ for $a + \epsilon$. 

\section{Statement of Results}\label{statements}

In the first part of this section, we give the main theorems which demonstrate the smoothing effect of the nonlinear flow. We then give two results which show some of the implications of smoothing for the global dynamics of the system. First, we state the theorem for the Zakharov system. 

\begin{thm}\label{Zthm}
Consider the Zakharov evolution \eqref{eq:tZ} on $\mathbb{R}^d$. If $d = 2,3$, assume that $r \geq - \frac12$  with $2s- r \geq \frac12$ and $r < s < r +1$. Then
\begin{align*}
u(t) - e^{it\Delta}u_0 \in C_{t}H^{s+\alpha}_x\\
n^\pm(t) -e^{\mp i t J}n_0^\pm \in C_{t} H^{r + \beta}_x 
\end{align*}
on the interval of existence as long as $\alpha < \min\{ \frac12, r -s + 1, r + 2 - \frac{d}2\}$ and $\beta < \min\{ 2s -r - \frac12, s - r\}$.
If $d \geq 4 $, assume $r > \frac{d-4}{4} $ and $2s - r > \frac{d-2}{2}$ with $r \leq s \leq r+1$. Then the same statement holds if $\beta < \min\{ 2s - r - \frac{d-2}2, s-r\}$. 
\end{thm}
The restrictions on $r$ and $s$ are necessary to ensure well-posedness of the equation. The values $\alpha$ and $\beta$ represent the smoothing effect. For instance, in dimensions $d = 2,3$, for initial data in $H^{\frac12} \times L^2$, the nonlinear part of the evolution lies in $H^{1-} \times H^{\frac12 - }$. A similar result holds for the Klein-Gordon-Schr\"{o}dinger system. 
\begin{thm}\label{KGSthm}
Consider the Klein-Gordon-Schr\"{o}dinger evolution \eqref{eq:tKGS} on $\mathbb{R}^d$. If $d = 2,3$, assume $s > -\frac14$ and $r > -\frac12$ with $2s -r  \geq - \frac32$ and $r - 2 < s <  r + 1$. Then we have 
\begin{align*}
u(t) - e^{it\Delta}u_0 \in C_{t}H^{s+\alpha}_x\\
v^\pm(t) -e^{\mp i t A}v_0^\pm \in C_{t} H^{r + \beta}_x 
\end{align*}
on the interval of existence as long as $\alpha < \min\{ \frac12, r -s + 1, r + 2 - \frac{d}2\}$ and $\beta < \min\{ 2s -r + \frac32, s - r + 2\}$. If $d \geq 4 $, assume $r > \frac{d-4}{4} $ and $2s - r > \frac{d-6}{2}$ with $r-2 \leq s \leq r+1$. Then the same statement holds if $\beta < \min\{ 2s - r - \frac{d-6}2, s-r+2\}$. 
\end{thm}

For the Klein-Gordon-Schr\"{o}dinger system, the smoothing effect on the wave part is much stronger than that on the Zakharov because of the lack of derivatives in the nonlinearity. For instance, in dimensions $d = 2,3$ and initial data in $L^2 \times L^2$, the nonlinear part is in $H^{\frac12-} \times H^{\frac32 -}$.  
The proof of these results is in Section \ref{smoothingProof}. It depends on the following new bilinear estimate for the Schr\"{o}dinger nonlinearity, together with the known local theory estimates for the wave equation nonlinearity. 

\begin{prop}\label{schrodinger_est}
Assume $d \geq 2$ and $b = \frac12+$ with $s, r > -\frac12$. Then the estimate 
\begin{equation} \label{eq:sch}
 \| uv\|_{X^{s+\alpha, b-1}} \lesssim \|u\|_{X^{s,b}}\|v\|_{X^{r,b}_\pm}
\end{equation}
holds for $\alpha< \min\{\frac12, r - s + 1, r + 2 - \frac{d}2\}$. The same result holds with the restricted versions of the norms. 
\end{prop}
We also state the estimates for the wave and Klein-Gordon nonlinear terms: 
\begin{prop}[\cite{BH, BHHT, GTV}]
Let $b = \frac12 +$. 
If $d = 2,3$, assume $s > -\frac 14$ with $2s - \sigma > \frac12$ and $ s > \sigma $. If $d \geq 4$, assume $2s - \sigma> \frac{d -2}2 $ and $\sigma \leq s$, $s \geq 0$. Then
\[ \| A(|u|^2)\|_{X^{\sigma, b-1}_\pm} \lesssim \|u\|_{X^{s,b}}^2. \]
The same result holds for the restricted versions of the norms.
\end{prop}
\begin{prop}[\cite{GTV, P1, P2}] \label{kg_est}
Let $b = \frac12 +$. If $d = 2,3$, assume $s > -\frac 14$ with $2s - \sigma >  -  \frac32$ and $ \sigma -2 < s $. If $d \geq 4$, assume that $2s - \sigma >  \frac{d -6}2 $ and  $\sigma - 2 \leq s$, $s \geq 0$. Then
\begin{equation}\label{eq:kg} \| A^{-1}(|u|^2)\|_{X^{\sigma, b-1}_\pm} \lesssim \|u\|_{X^{s,b}}^2. \end{equation} The same result holds for the restricted versions of the norms.
\end{prop}

We remark that a half derivative gain is the best that can be hoped for in the Schr\"{o}dinger evolution from the use of such bilinear estimates. To see that the bilinear estimate \eqref{eq:sch} fails for $\alpha > \frac12$, let $\hat{u} = \chi_{B_1}$ and $\hat{v} = \chi_{B_2}$, where 
\begin{align*}
B_1 &= \left\{ (\xi_1, \xi_2, \ldots, \xi_d, \tau) \in \mathbb{R}^{d+1} :  |\xi_{1} - N | < N^{-1}, \; |\xi_{i}| < 1 \text{ for } i \geq 2,\;  |\tau + N^2| < 1\right\}, \\
B_2 &= \left\{ (\xi_1, \xi_2, \ldots, \xi_d, \tau) \in \mathbb{R}^{d+1} : \qquad |\xi_{1} | < N^{-1},\; |\xi_{i}| < 1 \text{ for } i \geq 2, \quad\qquad |\tau | < 1\right\} 
\end{align*}
for some $N \gg 1$. Then $\|u\|_{X^{s,b}} \approx N^{s - \frac12}$ and $\|v\|_{X^{s,b}_\pm} \approx N^{-\frac12}$, while $\widehat{uv}$ is roughly $N^{-1} \chi_{B_1}$, so that $\|uv\|_{X^{s+\alpha,b}} \approx N^{s+ \alpha- \frac32}$. This can only be bounded by $\|u\|_{X^{s,b}} \|v\|_{X^{s,b}_\pm} \approx N^{s-1}$ when $\alpha \leq \frac12$. 


As an application of the smoothing estimate, we study the existence of global attractors for the dissipative Klein-Gordon-Schr\"{o}dinger evolution. The existence of global attractors for dissipative PDEs has been extensively researched (see e.g. \cite{Tem} and the references therein). Proofs generally use the dissipative property to obtain decay of solutions in the energy space, followed by a weak-convergence argument to show that all flows eventually enter a compact absorbing ball. This second step is particularly challenging on noncompact spaces such as $\mathbb{R}^d$, where proving compactness can be difficult. For the dissipative Klein-Gordon-Schr\"{o}dinger evolution on $\mathbb{R}^d$, $d \leq 3$, the existence of a global attractor was proved in \cite{LW}. In the following, we simplify the proof using our smoothing estimate.

With the addition of damping and forcing terms, the Klein-Gordon-Schr\"{o}dinger \eqref{eq:KGS} system becomes 
\begin{equation}\label{eq:dKGS}
 \begin{cases}
iu_t + \Delta u + i \gamma u = -uv + f , \quad x \in \mathbb{R}^d\\
v_{tt} + (-\Delta + 1) v + \delta v_t = |u|^2 + g.
\end{cases}
\end{equation}
We will be concerned with $d= 2,3$ and initial data $\big(u(x,0),v(x,0),v_{t}(x,0)\big)$ in the energy space $H^1\times H^1 \times L^2$ with damping coefficients $\gamma, \delta > 0$ and forcing terms $f, g \in H^1$. In the following, $U(t)$ will denote the evolution operator corresponding to \eqref{eq:tKGS}. Note that the notion of a global attractor is only reasonable when the system is globally well-posed. For the forced and weakly damped system, global well-posedness holds in the energy space $H^1 \times H^1 \times L^2$ by a minor modification of the nondissipative local theory arguments together with decay of the Hamiltonian energy (see \cite{ET2} for details). Before stating the result, we give some definitions. 

\begin{definition}[\cite{Tem}] \label{ga} A compact subset $\mathcal{A}$ of the phase space $H$ is called a global attractor for the semigroup $\{U(t)\}_{t \geq 0}$ if $\mathcal{A}$ is invariant under the flow of $U$ and 
\begin{align*}
\lim_{t \to \infty} d(U(t)u_0 ,\mathcal{A}) = 0 \text{  for every  } u_0 \in H.
\end{align*}
\end{definition}

Using energy estimates, it can be shown that all solutions eventually enter a bounded subset of $H^1 \times H^1 \times L^2$. Such a set is called an absorbing set for the evolution $U(t)$:

\begin{definition}[\cite{Tem}] \label{abset} A bounded subset $\mathcal{B}_0$ of $H$ is called absorbing if for any bounded $\mathcal{B} \subset H$, there exists a time $T = T(\mathcal{B})$ such that $U(t)\mathcal{B} \subset \mathcal{B}_0$ for all $t \geq T$. 
\end{definition}

The global attractor will be the $\omega$-limit set of $\mathcal{B}_0$, which is defined by 
\[ \omega(\mathcal{B}_0)  = \bigcap_{s \geq 0} \overline{\bigcup_{t \geq  s} U(t)\mathcal{B}_0}. \]
Notice that it is immediate that the existence of a global attractor implies the existence of an absorbing set. The converse does not hold, though; an absorbing set may not be invariant under the flow and need not be compact. A partial converse is true, however, and will be used to show that the $\omega$-limit set is indeed a global attractor. 

\begin{thm}[\cite{Tem}] \label{asymthrm}
Let $H$ be a metric space and $U(t)$ be a continuous semigroup from $H$ to itself for all $t \geq 0$. Assume that there is an absorbing set $\mathcal{B}_0$. If the semigroup $\{U(t)\}_{t \geq 0}$ is asymptotically compact, i.e. for every bounded sequence $\{x_k\} \subset H$ and every sequence $t_k \to \infty$, the set $\{U(t_k)x_k\}_k$ is relatively compact in $H$, then $\omega (\mathcal{B}_0)$ is a global attractor. 
\end{thm}

We will prove asymptotic compactness using a smoothing estimate for the dissipative system, yielding the following result. 

\begin{thm} \label{GAthrm}
The Klein-Gordon-Schr\"{o}dinger evolution in dimensions $d=2,3$ has a global attractor in $H^1 \times H^1 \times L^2$ which is compact in $H^{\frac32 -} \times H^{3-} \times H^{2-}$.   
\end{thm}
The existence of a global attractor is known \cite{LW}. However, the compactness statement appears to be new. We remark that the existence of a global attractor for the dissipative Zakharov system (without a mass term) on Euclidean spaces appears to be an interesting open problem. The methods we use cannot be applied to the Zakharov because of difficulties in controlling the low-frequency components of the wave equation. We also remark that our proof method also applies to \eqref{eq:dKGS} with forcing $f, g \in H^{-\frac12 +}$. In this case, we obtain a global attractor which is compact in $H^{\frac32 -} \times H^{\frac32+} \times H^{\frac12+}$. 

As a second application, we use a variant of the high-low decomposition method together with the smoothing estimate to obtain global existence for the Klein-Gordon-Schr\"{o}dinger equation in four dimensions. 

\begin{thm}
The Klein-Gordon-Schr\"{o}dinger evolution \eqref{eq:tKGS} is globally well-posed on $H^s \times H^r$ for $s,r > 9/10$ as long as  $\|u_0\|_{L^2} < \sqrt{2} C_1 C_2^2$, where $C_1$ and $C_2$ are the optimal constants in the four-dimensional $L^4$ and $L^{8/3}$ Gagliardo-Nirenberg-Sobolev inequalities respectively.
\end{thm}
The constraint on the norm of the $u_0$ is necessary to ensure that the energy functional is positive definite. The optimal constants in the Gagliardo-Nirenburg-Sobolev inequalities have been established by Weinstein \cite{W}. The proof of this result is in Section \ref{R4Proof}. 

\section{Proof of Theorems \ref{Zthm} \& \ref{KGSthm}}\label{smoothingProof}

In this section, we give the proof of the smoothing theorem for the Klein-Gordon-Schr\"{o}dinger flow. The proof for the Zakharov equation has the same structure; it is obtained by adding two derivatives to the wave nonlinearity which appears in the Klein-Gordon-Schr\"{o}dinger system. Since the calculations for the Zakharov equation are similar, they are omitted. 

Writing the solution to the transformed Klein-Gordon-Schr\"{o}dinger equation \eqref{eq:tKGS} in its Duhamel form yields
\begin{align*}
u(t) - e^{it\Delta}u_0 =-\frac12  \int_0^t e^{i(t-t')\Delta}\Big( u(v^+ + v^-)\Big) \d t' \\
v^\pm(t) - e^{\mp i t A} v^\pm_0 = \mp \int_0^t e^{\mp i (t-t') A} \Big( A^{-1}|u|^2\Big) \d t'.
\end{align*}
Let $\delta$ be the local existence time of the solution. Then on $[0,\delta]$ we have 
\begin{equation} \label{eq:localBound}
\|u\|_{X^{s,b}_\delta}  +  \|v^\pm \|_{X^{r,b}_{\pm,\delta}} \lesssim \|u_0\|_{H^s} + \|v^\pm_0\|_{H^r}.
\end{equation}

To control the Duhamel integral terms, we use the embeddings $X^{s,b} \hookrightarrow C^0H^{s}$ and $X^{r , b}_{\pm} \hookrightarrow C^0H^{r}$, which hold for $b > \frac12$,  along with following standard lemma. 
\begin{lem}[\cite{GTV}]
For $b \in (\frac12, 1]$, we have
\begin{align*}
\left\| \int_0^t e^{i(t-t')\Delta}F(t') \d t' \right\|_{X^{s,b}_\delta} &\lesssim  \|F\|_{X^{s,b-1}_{\delta}} \\
\left\|\int_0^t e^{\mp i (t-t')A}F(t') \d t' \right\|_{X^{r,b}_{\pm, \delta}} &\lesssim \|F\|_{X^{r,b-1}_{\pm,\delta}}. 
\end{align*}
\end{lem}
Using these estimates yields
\begin{align*}
\| u(t) - e^{it\Delta}u_0 \|_{L^\infty_{[0,\delta]}H^{s+ \alpha}} &\lesssim \|u v^+\|_{X^{s+\alpha, b-1}_\delta} + \|u v^-\|_{X^{s+\alpha, b-1}_\delta}\\
\|  v^\pm(t) - e^{\mp i t A} v^\pm_0\|_{L^\infty_{[0,\delta]}H^{r+ \beta}} &\lesssim \|  A^{-1}|u|^2 \|_{X^{r + \beta, b-1}_{\pm, \delta}}.
\end{align*}
Using the estimates from Propositions \ref{schrodinger_est} and \ref{kg_est}, we have 
\begin{align*}
\| u(t) - e^{it\Delta}u_0 \|_{L^\infty_{[0,\delta]}H^{s+ \alpha}} &\lesssim \|u\|_{X^{s,b}_{\delta}} \Big(\|v^+\|_{X^{r,b}_{+, \delta}} + \| v^-\|_{X^{r,b}_{-, \delta}}\Big)\\
\|  v^\pm(t) - e^{\mp i t A} v^\pm_0\|_{L^\infty_{[0,\delta]}H^{r+ \beta}} &\lesssim \| u\|^2_{X^{s}_{\delta}}.
\end{align*}
Using the local theory bound \eqref{eq:localBound}, we conclude
\begin{align*}
\| u(t) - e^{it\Delta}u_0 \|_{L^\infty_{[0,\delta]}H^{s+ \alpha}} &\lesssim \Big( \|u_0\|_{H^s} + \|v^\pm_0\|_{H^r} \Big)^2\\
\|  v^\pm(t) - e^{\mp i t A} v^\pm_0\|_{L^\infty_{[0,\delta]}H^{r+ \beta}} & \lesssim \Big( \|u_0\|_{H^s} + \|v^\pm_0\|_{H^r} \Big)^2.
\end{align*}
Repeating this process shows that the nonlinear part of the solution remains in $H^{s + \alpha} \times H^{r + \beta}$ for the full interval of existence. 

To prove continuity, write
\begin{align*}
 &\Big( u(t) - e^{-t\Delta}u_0 \Big)  -\Big(u(t + \epsilon) - e^{i(t+ \epsilon)\Delta}u_0 \Big) \\
&= \frac12  \int_0^{t+\epsilon} e^{i(t+\epsilon-t')\Delta}\Big( u(v^+ + v^-)\Big) \d t' - \frac12 \int_0^t e^{i(t-t')\Delta}\Big( u(v^+ + v^-)\Big) \d t' \\
&= \frac12 \big( e^{i\epsilon \Delta} - \operatorname{Id}\big)  \int_0^t e^{i(t-t')\Delta}\Big( u(v^+ + v^-)\Big) \d t' + \frac12  \int_t^{t+\epsilon} e^{i(t+\epsilon-t')\Delta}\Big( u(v^+ + v^-)\Big) \d t'
\end{align*}
The continuity follows by applying the estimates stated previously along with the continuity of $(u, v^\pm)$ in $H^s \times H^r$; see \cite{ET1}. Continuity of the nonlinear part of $v$ is proved in the same way.

\section{Proof of the Existence of a Global Attractor}\label{gaProof}

In this section, we use smoothing estimates to simplify the proof of the existence of a global attractor for the dissipative Klein-Gordon-Schr\"{o}dinger flow in two and three dimensions. To prove this result, we need to establish boundedness and asymptotic compactness of the flow. The boundedness follows from the energy equation; compactness is the challenging part. To prove this, we use boundedness to obtain a weakly convergent sequence of solutions. The energy equation is used to upgrade the weak convergence to strong convergence, yielding the desired compactness. The energy functional contains cubic terms which can easily be bounded using our smoothing result and the embedding $H^{\frac32+} \hookrightarrow L^{\infty}$. In the existing proof, an extensive argument, involving uniform estimates of the solution restricted to compact sets, is required to control these terms. 

First we establish a weak continuity result for the evolution operator which will be needed to work with the energy equations. A slightly weaker form of the following lemma is in \cite[Lemma 3.1]{LW}. 

\begin{lem}\label{weakcont}
 Let $d = 2,3$. Let $\S(t)$ denote the semigroup operator for \eqref{eq:dKGS}, and let $L(t)$ denote the linear part of the semigroup operator. If $(u_0^n,v_0^n,w_0^n) \rightharpoonup (u_0,v_0,w_0)$ weakly in $H^1 \times H^1 \times L^2$, then for any $T >0$,  
 \begin{align*}
  L(t)(u_0^n,v_0^n,w_0^n) &\rightharpoonup L(t)(u_0,v_0,w_0) \quad\text{weakly in}\quad L^2([0,T], H^1 \times H^1 \times L^2) \\
  \bigl[\S(t) - L(t) \bigr] (u_0^n,v_0^n,w_0^n) &\rightharpoonup \bigl[\S(t) - L(t)\bigr] (u_0,v_0,w_0) \quad\text{weakly in}\quad L^2([0,T], H^{\frac32 - } \times H^{3-} \times H^{2-}). 
 \end{align*}
 Furthermore, we have pointwise weak convergence: for any $t \in [0,T]$, 
 \begin{align*}
  L(t)(u_0^n,v_0^n,w_0^n) &\rightharpoonup  L(t)(u_0,v_0,w_0) \quad \text{weakly in} \quad H^1 \times H^1 \times L^2 \\
  \bigl[\S(t)-L(t) \bigr](u_0^n,v_0^n,w_0^n) &\rightharpoonup \bigl[ \S(t)-L(t) \bigr] (u_0,v_0,w_0) \quad\text{weakly in}\quad H^{\frac32 - } \times H^{3-} \times H^{2-}. 
 \end{align*}
\end{lem}

\begin{proof}
 The statements for the linear part of the flow can be verified using the Fourier multiplier representation of the linear solutions. To work on the nonlinear part, we transform the equation \eqref{eq:dKGS}. Set $\tilde{f} = (1-\Delta)^{-1}f$ and $\tilde{g} = (1-\Delta)^{-1}g$. Let $\tilde{u} = u + \tilde{f}$ and $\tilde{v} = v - \tilde{g}$ with $w = a v + v_t$. We will choose $0< a \ll 1$ later. This transformation yields
\begin{equation}\label{eq:dampedtransf}
 \begin{cases}
i\tilde{u}_t + \Delta \tilde{u}+  i \gamma \tilde{u} = - (\tilde{u} - \tilde{f})(\tilde{v} + \tilde{g}) + (1 + i \gamma )\tilde{f}\\
\tilde{v}_t +  a \tilde{v} + a \tilde{g} =  w \\
w_t + (\delta - a) w + \Bigl(1 + a(a-\delta) - \Delta\Bigr)\tilde{v} = |\tilde{u} - \tilde{f}|^2 - a(a-\delta)\tilde{g}.
\end{cases}
\end{equation}
The transformation allows us to replace the $H^1$ forcing terms by $H^3$ forcing terms, in exchange for more complex nonlinearities. The introduction of $w$ is convenient for energy calculations. 

Consider the homogeneous linear system
\begin{equation} \label{eq:lin}
\begin{cases}
ip_t + \Delta p+  i \gamma p = 0\\
q_t + a q = r \\
r_t + (\delta - a) r + \Bigl(1 + a(a-\delta) - \Delta\Bigr)q = 0.
\end{cases}
\end{equation}
In the following, we abuse notation and let $L$ also denote the semigroup associated with this equation. The nonlinear parts $(U,V,W) = (\tilde{u}-p, \tilde{v}-q, w - r)$ satisfy
\begin{equation} \label{eq:nonlin}
\begin{cases}
iU_t + \Delta U +  i \gamma U = - (U + p - \tilde{f})(V + q + \tilde{g}) + (1+ i \gamma)\tilde{f}  \\
V_t + a V + a \tilde{g} = W \\
W_t + (\delta - a) W + \Bigl(1 + a(a-\delta) - \Delta\Bigr)V = |U + p - \tilde{f}|^2  - a(a-\delta) \tilde{g},
\end{cases}
\end{equation}
with zero initial data. 
Just as in Section \ref{smoothingProof}, we can use bilinear estimates, together with the smoothness of $\tilde{f}$ and $\tilde{g}$, to conclude that $(U, V, W)\in H^{\frac32 - } \times H^{3-} \times H^{2-}$ for initial data $H^1 \times H^1 \times L^2$ in dimensions $d = 2,3$. 

We will show that every subsequence of $\bigl[\S(t) - L(t) \bigr] (u_0^n, v_0^n, w_0^n)$ has a further subsequence which converges weakly to the solution of the KGS, which implies that the full sequence converges weakly to that solution. 

Note that the sequence $(u_0^n, v_0^n, w_0^n)$ is uniformly bounded in the energy space. Denote by $(p^n, q^n ,r^n)$ the solution to linear system \eqref{eq:lin} with initial data $(u_0^n,v_0^n,w_0^n)$. Let $(U^n,V^n,W^n)$ be the nonlinear part of the flow. For the nonlinear part, smoothing estimates along with the uniform bound on the initial data imply that, for any $T>0$, 
\begin{multline} \label{eq:bds}
\Bigl\{(U^n,V^n,W^n)\Big\} \;\text{  is bounded in  } \\
C([-T,T], H^{\frac32 - } \times H^{3-} \times H^{2-}) \cap  C^1([-T,T], H^{ -\frac12 - } \times H^{ 2- } \times H^{1-}). 
\end{multline} 
This has several implications: \begin{enumerate}[{(i)}]
\item \label{itm:weakstar} The Banach-Alaoglu theorem implies weak* convergence of a subsequence  of 
\[\big\{(U^n,V^n,W^n)\big\} \text{\quad in \quad } L^\infty([-T,T], H^{\frac32 - } \times H^{3-} \times H^{2-}). \]
\item \label{itm:weakloc} The Arzela-Ascoli theorem 
implies that $\big\{(U^n,V^n,W^n)\big\}$ is precompact in $C([-T,T], H_{\operatorname{loc}}^{ -\frac12 - } \times H_{\operatorname{loc}}^{ 2-} \times H_{\operatorname{loc}}^{1-})$. By interpolation between this and \eqref{eq:bds}, we find  strongly convergent subsequence of 
\[ \big\{(U^n,V^n,W^n)\big\} \text{\quad in \quad} C([-T,T], H_{\operatorname{loc}}^{ \frac32 -} \times H_{\operatorname{loc}}^{3-} \times H_{\operatorname{loc}}^{2-}). \] 
\end{enumerate}
Similar statements hold for the linear parts $(p^n, q^n, r^n)$ in $H^1 \times H^1 \times L^2$. By passing to a further subsequence, we obtain a sequence, which we still call $\bigl\{(U^n, V^n,W^n)\bigr\}$, which is weak* convergent in $L^\infty([-T,T], H^{\frac32 - } \times H^{3-} \times H^{2-})$ and strongly convergent in $C([-T,T], H_{\operatorname{loc}}^{ \frac32 -} \times H_{\operatorname{loc}}^{3-} \times H_{\operatorname{loc}}^{2-})$. Denote the limit by $(U,V,W)$.  

To see that the limit is a distributional solution, multiply the equations for $(U^n, V^n,W^n)$  by an arbitrary test function $\phi \in C^\infty_c([-T,T] \times \mathbb{R}^d)$, integrate in space and time, and take the limit in $n$. For $U^n$, we have
\begin{align*} 
 &\iint  \biggl[ -i  U\phi_t  + U \Delta \phi  + i \gamma  U \phi + \phi  \Big((U + p - \tilde{f})(V + q + \tilde{g}) - (1+ i \gamma) \tilde{f} \Big) \biggr] \d x \d t \\
= \lim_{n \to \infty}&\iint  \biggl[ -i  U^n\phi_t+ U^n \Delta \phi +  i \gamma U^n\phi  
   + \phi \Big(  (U^n + p^n - \tilde{f})(V^n + q^n + \tilde{g}) - (1+ i \gamma) \tilde{f} \Big) \biggr] \d x \d t= 0. 
\end{align*}
The equality is a consequence of the local strong convergence (\ref{itm:weakloc}) of $U^n$ and $V^n$ and strong local convergence of $p^n$ and $q^n$ in $C([-T,T], H^{1-})$. To verify the limit for the nonlinear term, note that 
\begin{align*}
& \left| \iint \phi \Bigl[(U^n + p^n - \tilde{f})( V^n + q^n + \tilde{g}) - (U + p -\tilde{f})(V + q +\tilde{g})\Bigr] \d x \d t   \right|\\
= &\left|\iint \phi\Bigl[ (U^n + p^n - \tilde{f})\bigl[ V^n - V + q^n - q \bigr] + \bigl[U^n - U + p^n -  p \bigr](V + q + \tilde{g})   \Bigr] \d x\d t \right|  \\
\leq & \|\phi\|_{L^\infty_{x,t}} \left| \iint_{\operatorname{supp}\phi}  \Bigl[ (U^n + p^n - \tilde{f})\bigl[ V^n - V + q^n - q \bigr] + \bigl[U^n - U + p^n -  p \bigr](V + q + \tilde{g})   \Bigr] \d x\d t \right|,
\end{align*}
which decays by local strong convergence. For $V^n$ and $W^n$, we have 
\begin{align*}
 \iint \biggl[ -V \phi_t  + \phi \Big( a V  + a \tilde{g} - W \Big) \biggr] \d x \d t 
= \lim_{n \to \infty} &\iint \biggl[ -V^n \phi_t  +  \phi \Big( a V^n  + a \tilde{g}  - W^n \Big) \biggr] \d x \d t  = 0 
\end{align*}
and
\begin{align*} 
 & \iint  \biggl[  - W\phi_t  - V \Delta \phi + \phi \Big( (\delta - a)  W   + \bigl(1 + a(a-\delta)\bigr) V \Bigr) - \phi \Bigl(|U + p + \tilde{f}|^2 - a(a - \delta) \tilde{g}\Big)\biggr] \d x \d t \\
&=  \lim_{n \to \infty} \iint \biggl[  W^n \phi_t - V^n \Delta \phi  + \phi\Big((\delta - a)W^n   +\bigl(1 + a(a-\delta)\bigr) V^n \Big)  \\
&\hspace{2.65in} -   \phi \Big(|U^n+ p^n + \tilde{f}|^2  - a(a - \delta) \tilde{g} \Big) \biggl] \d x \d t  = 0. 
\end{align*}
Again, convergence of the nonlinear terms follows from strong local convergence. Thus $(U,V,W)$ is a distributional solution of \eqref{eq:nonlin} with $(U(0),V(0),W(0)) = (0,0,0)$. Furthermore, by the weak* convergence (\ref{itm:weakstar}), we see that $(U,V,W)$ is in the uniqueness class $C([-T,T], H^{\frac32-} \times H^{3-} \times H^{2-})$. Thus $(U,V,W) = \bigl[ \S(t) - L(t) \bigr](u_0,v_0,w_0)$. The weak* convergence (\ref{itm:weakstar}) implies weak convergence in $L^2([-T,T], H^{\frac32-} \times H^{3-} \times H^{2-})$ as desired since $L^2([-T,T], H^{\frac32-} \times H^{3-} \times H^{2-})$ is contained in the dual of $C([-T,T], H^{\frac32-} \times H^{3-} \times H^{2-})$. 

To show pointwise weak convergence, fix a $t_0 \in [0,T]$. By again applying the Banach-Alaoglu theorem and passing to a further subsequence if necessary, we can ensure that the convergence described above still holds, along with weak $H^{\frac32-} \times H^{3-} \times H^{2-}$ convergence of $\bigl[\S(t_0) - L(t_0) \bigr](u_0^n,v_0^n,w_0^n)$, say to $(u^*,v^*,w^*)$. Recall that we have shown weak* convergence of $\bigl[\S(t) - L(t) \bigr](u_0^n,v_0^n,w_0^n)$ to $(U,V,W) = \bigl[\S(t) - L(t)\bigr](u_0,v_0,w_0)$ in $C([0,T], H^{\frac32 -} \times H^{3-} \times H^{2-})$. Thus we have $(u^*,v^*,w^*) = \bigl[\S(t) - L(t) \bigr](u_0,v_0,w_0)$. 
\end{proof}

In the remainder of this section, we work with the following transformation of \eqref{eq:dKGS}:
\begin{equation} \label{eq:dKGStransf}
 \begin{cases}
  iu_t + \Delta u + i \gamma u = -uv + f \\
  v_t + a v = w \\
  w_t + (\delta - a) w + \Bigl( 1 + a(a - \delta) - \Delta \Bigr) v = |u|^2 + g.
 \end{cases}
\end{equation}
Again, let $\S(t)$ denote the semigroup operator for \eqref{eq:dKGStransf}, and let $L(t)$ denote the linear flow operator. The evolution \eqref{eq:dKGStransf} has the absorbing ball property in dimensions $d=2,3$. For a proof of this, see \cite[Section 2]{LW}. Thus to obtain a global attractor, it suffices to prove that the evolution is asymptotically compact -- that is, that for every sequence of initial data $\{(u_0^n,v_0^n,w_0^n)\}_n$ in the energy space with corresponding solutions $\{(u^n,v^n,w^n)\}_n$ and every sequence of times $t_n \to \infty$, the sequence $\{(u^n(t_n) ,v^n(t_n),w^n(t_n))\}_n$ has a convergent subsequence in the energy space.

Let $(u_n,v_n,w_n) \in H^1 \times H^1 \times L^2$ be a sequence of initial data. We may assume that the data lies within the absorbing ball. Also let $t_n \to \infty$ be a sequence of times. The Banach-Alaoglu theorem implies that the sequence $\S(t_n)(u_n,v_n, w_n)$ has weakly convergent subsequence in $H^1 \times H^1 \times L^2$. Smoothing estimates together with bounds on the initial data imply that the nonlinear parts are bounded in $H^{\frac32-} \times H^{3-} \times H^{2-}$. Thus we may choose a subsequence such that these nonlinear parts $\S(t_n)(u_n,v_n, w_n) - L(t_n)(u_n,v_n, w_n)$ also converge weakly in $H^{3/2-} \times H^{3-} \times H^{2-}$. Since the linear part decays to zero, these two limits must be equal. Call the limit $(u,v,w)$. 

For any $T$, we can, by passing to a further subsequence, conclude that $\S(t_n - T)(u_n,v_n, w_n)$ converges weakly in $H^1 \times H^1 \times L^2$ and that $\S(t_n - T)(u_n,v_n, w_n) - L(t_n - T)(u_n,v_n, w_n)$ converges weakly in $H^{3/2-} \times H^{3-} \times H^{2-}$. Again, dissipative decay of the linear part implies that the two limits are equal; we denote the limit by $(u_T,v_T,w_T)$. Weak continuity of the semigroup (Lemma \ref{weakcont}) implies that $\S(T)(u_T, v_T, w_T) = (u,v,w)$. Note that by a diagonalization argument, we can obtain such weak convergence of $\S(t_n - T)(u_n,v_n, w_n)$ and $\S(t_n - T)(u_n,v_n, w_n) - L(t_n - T)(u_n,v_n, w_n)$ as above for a countable set of $T$ simultaneously, e.g. $\{T \in \mathbb{N}\}$. This will be important later when we take $T \to \infty$. 

The $L^2$ law for the evolution of $\S(t)$ gives 
\begin{align*}
 \| \S(t_n)&u_n \|_{L^2}^2 = e^{-2\gamma T} \| \S(t_n - T) u_n\|_{L^2}^2 - 2\operatorname{Re} i \int_0^T  e^{2 \gamma (s-T)} \Bigl\lb\S(t_n - T +s) u_n ,\;  f  \Bigr\rb_{L^2_x} \d s, \\
 \| \S(T)&u_T \|_{L^2}^2 = e^{-2\gamma T} \| u_T\|_{L^2}^2 - 2 \operatorname{Re} i \int_0^T e^{2 \gamma (s-T)} \Bigl\lb \S(s) u_T,\;  f \Bigr\rb_{L^2_x} \d s. 
\end{align*}
Combining the two equations yields
\begin{align*}
 \| \S(t_n)u_n \|_{L^2}^2 - \| \S(T)u_T \|_{L^2}^2 &= e^{-2\gamma T} \Bigl( \| \S(t_n - T) u_n\|_{L^2}^2 - \| u_T\|_{L^2}^2 \Bigr) \\
 &+ 2 \operatorname{Re} i \int_0^T \hspace{-4pt} e^{2 \gamma (s-T)} \Bigl\lb \S(s) u_T -\S(t_n - T +s) u_n,\; f \Bigr\rb_{L^2_x}\d s.
\end{align*}
The first term on the right-hand side can be made arbitrarily small by increasing $T$ since the $u_n$ are uniformly bounded in $L^2$.
The second term decays to zero as $n \to \infty$ by the weak continuity of $\S$ in $L^2_tH^1_x$. Thus we conclude that
\[ \limsup_{n \to \infty} \Bigl[ \| \S(t_n)u_n \|_{L^2}^2 - \| \S(T)u_T \|_{L^2}^2 \Bigr] = \limsup_{n \to \infty} \Bigl[ \| \S(t_n)u_n \|_{L^2}^2 - \| u \|_{L^2}^2 \Bigr] \leq 0. \]
With the weak convergence of $S(t_n)u_n$ to $u$, this implies that $\S(t_n) u_n \to u$ strongly in $L^2$. 

Now consider the full $\dot{H}^1 \times H^1 \times L^2$ energy equation. Define the energy functional $ H =  H(u_0,v_0,w_0)(t)$ as follows:
\begin{align*}
 H = 2 \| \nabla \S(t)u_0 \|_{L^2}^2 + \Bigl(1 + a(a-\delta)\Bigr)\| \S(t)v_0\|^2_{L^2}  &+ \| \nabla \S(t)v_0\|_{L^2}^2 + \|\S(t)w_0\|_{L^2}^2 \\
 &- 2 \int  |\S(t)u_0|^2\S(t)v_0 \d x  +  4 \int f \overline{S(t)u_0} \d x. 
\end{align*}
Then the time derivative $\d H/ \operatorname{d} t$ is given by
\begin{align*} 
-4 \gamma \| \nabla \S(t)u_0 \|_{L^2}^2 -2a\Bigl(1 + a(a-\delta)\Bigr)\| \S(t)v_0\|^2_{L^2} - 2a \| \nabla \S(t)v_0\|_{L^2}^2 - 2(\delta - a) \|\S(t)w_0\|_{L^2}^2 \\
+(4 \gamma + 2a) \int |\S(t)u_0|^2\S(t)v_0 \d x  - 4\gamma \operatorname{Re} \int f \overline{S(t)u_0} \d x + 2 \int g S(t)w_0 \d x .
\end{align*}
This implies that 
\[ H(u_n,v_n,w_n)(t_n) - H(u_T,v_T,w_T)(T) = \text{I} + \text{II} + \text{III} + \text{IV} + \text{V}, \] 
where
\begin{align*}
 \text{I} = &\;  e^{-2aT}\Bigl( H(u_n,v_n,w_n)(t_n-T) - H(u_T,v_T,w_T)(0) \Bigr) \\
 \text{II} = & -4(\gamma - a) \int_0^T e^{2a(s-T)} \Bigl[ \|\nabla \S(t_n-T+s)u_n\|^2_{L^2} - \|\nabla \S(s)u_T\|^2_{L^2} \Bigr] \d s \\
 &-2(\delta - 2a) \int_0^T e^{2a(s-T)} \Bigl[ \|\S(t_n - T + s)w_n\|_{L^2}^2  - \|S(s)w_T\|^2_{L^2} \Bigr] \d s \\
 \text{III} = & \; 2(2\gamma - a) \int_0^T \hspace{-4pt}\int e^{2a(s-T)} \Bigl[ \big|\S(t_n - T + s)u_n\big|^2 \S(t_n-T+s)v_n - \big|\S(s)u_T\big|^2 \S(s)v_T \Bigr] \d x \d s \\
 \text{IV} = & -4 (\gamma - 2a) \operatorname{Re} \int_0^T  e^{2a(s-T)} \Bigl\lb \S(t_n - T + s) u_n - S(s) u_T,\; f \Bigr\rb_{L^2_x} \d s \\
 \text{V} = & \; 2 \int_0^T  e^{2a(s-T)} \Bigl\lb \S(t_n - T + s) w_n - S(s) w_T,\; g \Bigr\rb_{L^2_x} \d s. 
\end{align*}
The term $\text{I}$ is negligible for large $T$. For $\text{II}$, weak convergence implies that 
\[ \liminf_{n \to \infty} \|\nabla \S(t_n-T+s)u_n\|^2_{L^2} - \|\nabla \S(s)u_T\|^2_{L^2} \geq 0,\]
\[ \liminf_{n \to \infty} \|\S(t_n-T+s)w_n\|^2_{L^2} - \| \S(s)w_T\|^2_{L^2} \geq 0\]
for each $s$, so the $\limsup$ over $n$ of $\text{II}$ is nonpositive.
Write the integral in $\text{III}$ as 
\begin{align*}
 & \int_0^T \int e^{2a(s-T)} \Bigl[ \big|\S(t_n - T + s)u_n\big|^2 - \big|\S(s)u_T\big|^2\Bigr]L(t_n-T+s)v_n \d x \d s \\
 + & \int_0^T \int e^{2a(s-T)} \Bigl[ \big|\S(t_n - T + s)u_n\big|^2 - \big|\S(s)u_T\big|^2\Bigr]\bigl[\S - L\bigr] (t_n-T+s)v_n \d x \d s \\
 + & \int_0^T \int e^{2a(s-T)} \big|\S(s)u_T\big|^2 \Bigl[ \S(t_n-T+s)v_n - \S(s)v_T \Bigr] \d x \d s.
\end{align*}
To see that the first line vanishes in the limit, apply the $L^3$ Gagliardo-Nirenberg inequality $\| h\|_{L^3} \lesssim \|\nabla h \|_{L^2}^{d/6} \| h \|_{L^2}^{(6-d)/6}$ with the fact that $L(t_n -T+s)v_n \to 0$ uniformly in $H^1$. 
For the second line, extract $\bigl[\S - L\bigr] (t_n-T+s)v_n$ in the $H^{3-} \hookrightarrow L^\infty$ norm and use the strong $L^2$ convergence of $\S(t_n)u_n$ to $\S(T)u_T$ and strong continuity of $\S(s-T)$. The last line decays by weak continuity of $\S(s)$ since $|\S(s)u_T|^2$ is an $L^2$ function by the Gagliardo-Nirenberg inequality $\| h\|_{L^4} \lesssim \|\nabla h \|_{L^2}^{d/4} \| h \|_{L^2}^{(4-d)/4}$. The remaining terms $\text{IV}$ and $\text{V}$ vanish in the limit by weak continuity of the semigroup. 

Thus we conclude that 
\[ \limsup_{n \to \infty} \Bigl[ H(u_n,v_n,w_n)(t_n) - H(u_T,v_T,w_T)(T)\Bigr] = \limsup_{n \to \infty} \Bigl[ H(u_n,v_n,w_n)(t_n) - H(u,v,w)(0)\Bigr] \leq 0. \]
This, together with the weak convergence of $S(t_n)(u_0^n, v_0^n, w_0^n)$ to $(u,v,w)$ in $H^1 \times H^1 \times L^2$, implies that $S(t_n)(u_0^n, v_0^n, w_0^n)$ converges strongly to $(u,v,w)$ in $H^1 \times H^1 \times L^2$. This completes the proof of asymptotic compactness, and thus of the existence of a global attractor. 
\section{Proof of Global Existence in $\mathbb{R}^4$}\label{R4Proof}

In this section, we prove global existence for the Klein-Gordon-Schr\"{o}dinger system in four dimensions. We work with the form of the equation given in \eqref{eq:tKGS} in dimension $d=4$. In the following, we drop the $\pm$ superscripts on $n$ to simplify the notation. Suppose we have $(u_0, n_0) \in H^s \times H^r$ for some $s,r > 9/10$ with $\|u_0\|_{L^2}$ small. Fix $T$ large. We wish to show that the solution $(u,n)$ exists on $[0,T]$. To do so, we decompose the solution into two parts: one with low-frequency initial data and one with the complementary high-frequency data. Specifically, recall that $A = (1- \Delta)^{1/2}$ and write $u = \phi  + \mu$ and $n= \psi + \lambda$, where 
\begin{equation} \label{eq:low}
\begin{cases}
i \phi_t + \Delta \phi = -\frac12 \operatorname{Re}(\psi) \phi\\
i\psi_t  \mp A\psi = \mp {A}^{-1}{|\phi|^2},
\end{cases}
\end{equation}
\begin{equation}\label{eq:high}
\begin{cases}
i \mu_t + \Delta \mu = -\frac12\operatorname{Re}(\lambda + \psi)\mu -\frac12 \operatorname{Re}(\lambda)\phi \\
i\lambda_t  \mp A\lambda = \mp {A}^{-1}|\mu|^2 \mp 2 \operatorname{Re}   A^{-1} \mu \overline{\phi}.
\end{cases}
\end{equation}
The initial data for these two systems is $(\phi_0,\psi_0) = (P_{\leq N} u_0,  P_{\leq N}n_0)$ and $(\mu_0, \lambda_0) = (u_0 - \phi_0 , n_0 - \psi_0)$, where $P_{\leq N}$ is the projection onto Fourier modes less than $N$. We will allow these equations to evolve a local-theory time step $\delta$. Then we add the nonlinear part of  $(\mu, \lambda)$ to $(\phi, \psi)$ and start again, i.e. evolve \eqref{eq:low} and \eqref{eq:high} another local time step with initial data 
\begin{align*} (\phi_1, \psi_1) &= \Big(\phi(\delta) + \bigl[ \mu(\delta) - e^{i\delta\Delta}\mu_0 \bigr] ,\; \psi(\delta) + \bigl[ \lambda(\delta) - e^{\mp i\delta A}\lambda_0 \bigr]\Big) \\
(\mu_1, \lambda_1)&=(e^{i\delta\Delta}\mu_0,\; e^{\mp i \delta A}\lambda_0).
\end{align*}
To iterate this process, we use smoothing estimates to show that the nonlinear part of $(\mu, \lambda)$ is in $H^1 \times H^1$ and that relevant norms do not grow too rapidly, so that a uniform time step $\delta$ can be used to cover $[0,T]$. 

We will need the following observations, which hold for $s_0 \leq s \leq 1$ and $r_0 \leq r \leq 1$ to carry out the local theory estimates:
\begin{align}\label{eq:highlo}
\begin{split}
\|\phi_0 \|_{H^1} &\leq N^{1-s} \|u_0\|_{H^s} \lesssim N^{1-s}   \qquad \;\;\, \|\psi_0\|_{H^1} \leq N^{1-r} \|n_0\|_{H^r} \lesssim N^{1-r}\\
\|\mu_0\|_{H^{s_0}} &\leq  N^{s_0 -s} \|u_0\|_{H^s} \lesssim N^{s_0 -s }   \qquad  \|\lambda_0\|_{H^{r_0}} \leq N^{r_0 -r} \|n_0\|_{H^r} \lesssim N^{r_0 - r}.
\end{split}
\end{align} 
In the following, we let $m = \min\{s, r\}$ and $s_0 = r_0 = \frac12+$. 

In four dimensions, the relevant nonlinear estimates for the local theory are in \cite{GTV}. In particular, from \cite[Lemma 3.4]{GTV}, we obtain
\[ \| un\|_{X^{k, -\frac12 + \epsilon}} \lesssim T^\theta \|u\|_{X^{k, \frac12 + \epsilon}}\|n\|_{X^{\ell, \frac12 + \epsilon}_\pm}\]
where $\theta = \min\{ \frac12, \frac\ell2 - (1 - \frac1{1 + 2\epsilon}) \}$ and $k \leq \ell + 1 - 2\epsilon$. From \cite[Lemma 3.5]{GTV}, we have also
\[ \| A^{-1}|u|^2 \|_{X^{\ell , -\frac12 + \epsilon}_\pm} = \| |u|^2 \|_{X^{\ell-1, -\frac12 + \epsilon}_\pm} \lesssim T^\theta \|u\|^2_{X^{k, \frac12 + \epsilon}}\]
as long as $2k - \ell + 1 > 0$ with $\theta =(2k - \ell + 1)/2 - 2 \epsilon/(1 + 2 \epsilon)$.
Thus a time step
\[ \delta \lesssim N^{-2(1-m)/{r_0}-}\lesssim \left\lb  \|\mu_0\|_{H^{s_0}} + \|\lambda^\pm_0\|_{H^{r_0}} + \| \phi_0\|_{H^1} + \|\psi_0 ^\pm\|_{H^1}  \right\rb^{-2/{r_0}- } .\] 
yields local existence for \eqref{eq:low} in $H^1 \times H^1$ on $[0,\delta]$ with 
\[ \| \phi\|_{X^{1,\frac12+}_\delta} + \|\psi\|_{X^{1,\frac12+}_{\pm, \delta}} \lesssim \|\phi_0 \|_{H^1} + \|\psi_0\|_{H^1} \lesssim N^{1-m}, \]
and local existence for \eqref{eq:high} in $H^{s_0} \times H^{r_0}$ with
\[ \| \mu\|_{X^{s_0,\frac12+}_\delta} + \|\lambda\|_{X^{r_0,\frac12+}_{\pm, \delta}} \lesssim \|\mu_0 \|_{H^{s_0}} + \|\lambda_0\|_{H^{r_0}} \lesssim N^{\max\{ s_0 -s , r_0 -r \}} = N^{1/2 - m +}. \]

Next write $\mu = e^{it\delta\Delta}\mu_0 + w(t)$ and $\lambda = e^{\mp itA}\lambda_0 + z(t)$, where $w$ and $z$ are the Duhamel terms
\begin{align*}
w(t) &= -\frac12 \int_0^t e^{i(t-s)\Delta}\Bigl[\operatorname{Re}(\lambda + \psi)\mu + \operatorname{Re}(\lambda)\phi\Bigr] \d s, \\
z(t) &=  \int_0^t e^{\mp i (t-s)A} \Bigl[ \mp {A}^{-1}|\mu|^2 \mp 2 \operatorname{Re}   A^{-1} \mu \overline{\phi} \Bigr] \d s. 
\end{align*}
Then for $r_0 = s_0 = \frac12 +$ and $m = \min\{s,r\}$, we have, using the estimates \eqref{eq:sch} and \eqref{eq:kg} along with the local theory bounds, 
\begin{align}\label{eq:hilononlin}
\begin{split}
\|w\|_{L^\infty_{[0,\delta]}H^{1}} &\lesssim \|w\|_{X^{1, \frac12+}_\delta} \lesssim \| \mu \|_{X^{s_0, \frac12 +}_\delta}\Big(\|\lambda\|_{X^{r_0, \frac12 +}_{\pm ,\delta}} + \| \psi\|_{X^{1, \frac12 +}_{\pm ,\delta}} \Big) + \|\phi\|_{X^{1, \frac12 +}_\delta}\|\lambda\|_{X^{r_0, \frac12 +}_{\pm ,\delta}} \\
&\lesssim N^{\max\{s_0 -s, r_0 -r\}}N^{1-m} = N^{3/2 - 2m+} \\
\|z\|_{L^\infty_{[0,\delta]}H^{1}} &\lesssim \| w\|_{X^{1, \frac12 +}_{\pm ,\delta}} \lesssim \| \mu\|_{X^{s_0, \frac12+}_\delta}\Big(\|\mu \|_{X^{s_0, \frac12+}_\delta} + \|\phi\|_{X^{1,\frac12+}_\delta} \Big)\\
&\lesssim N^{\max\{s_0 -s, r_0 -r\}}N^{1-m} = N^{3/2 - 2m+}.
\end{split}
\end{align}

To iterate, we must ensure that estimates \eqref{eq:highlo} and \eqref{eq:hilononlin} remain valid for each time step. This is immediate for the $(\mu, \lambda)$ initial data, since it is always simply a linear flow. Thus proving the requisite bounds amounts to showing that the $H^1 \times H^1$ norm of $(\phi, \psi)$ is bounded by $N^{1-m}$ over each time step. To do so, we use the $L^2$ conservation and the Hamiltonian energy. Notice that the initial data for $\phi$ at time step $k$ is $u(k \delta) - e^{ik\delta\Delta}\mu_0$, and the $L^2$ conservation gives $\| u - e^{ik\delta \Delta} \mu_0\|_{L_2} \leq \|u_0\|_{L_2}(1 + N^{-s})$. Thus we have uniform control over $\|\phi\|_{L^2}$. To control the remaining components of the $H^1 \times H^1$ norm, use the Hamiltonian
\[E(u, n) = \|An\|^2_{L^2} + 2 \|\nabla u \|^2_{L^2} - 2\int |u|^2 \operatorname{Re} (n )\d x.\] 
This is conserved for the flow of \eqref{eq:low}, so we need only check that it does not grow too much due to the addition of the nonlinear terms. Using the Gagliardo-Nirenberg inequality and Cauchy-Schwarz, the increment of the energy is bounded as follows, where the norms are all evaluated at time $\delta$: 
\begin{align*}
\Big|& E(\phi(\delta) + w(\delta), \psi^\pm(\delta) + z^\pm(\delta)) - E(\phi(\delta), \psi^\pm(\delta)) \Big| \\
&\lesssim \|Az^\pm\|_{L^2}\Big(\|Az^\pm\|_{L^2} + 2 \|An^\pm\|_{L^2} \Big) + 2 \|\nabla w\|_{L^2} \Big(\|\nabla w \|_{L^2} + 2 \|\nabla \phi\|_{L^2} \Big) \\
&+\|\nabla z^\pm\|_{L^2}\|\phi + w\|_{L^2} \|\nabla\big(\phi + w\big)\|_{L^2} \\
&+ \|\nabla \psi^\pm\|_{L^2} \Big( \|w\|_{L^2} \|\nabla w\|_{L^2} + \|\phi+w\|_{L^2}\|\nabla w \|_{L^2} \Big)
\end{align*}
Noting that $\|\phi + w \|_{L^2} = \|u - e^{it\Delta}\mu_0\|_{L^2} \lesssim 1$ and $\|w\|_{L^2} \lesssim 1$, this quantity can be controlled by $N^{3/2 - 2m +}N^{1-m}$, so for $E$ to remain bounded by $N^{2(1-m)}$, we require
\[ N^{3/2 - 2m +}N^{1-m}N^{2(1-m)/r_0+} \lesssim N^{2(1-m)}, \]
which holds if $m > 9/10$. 
To complete the proof, we need to show that the energy controls the $\dot{H}^1 \times H^1$ norm at each time step. This depends on the smallness assumption on $u$ in $L^2$. By the Gagliardo-Nirenberg inequality, we have 
\[ \left| \int |\phi|^2 \psi \d x \right| \leq \|\phi\|^2_{L^{8/3}} \|\psi\|_{L^4} \leq C_1C_2^2 \|\phi\|_{L^2}\|\nabla \phi \|_{L^2} \|\nabla \psi\|_{L^2} \leq  C_1C_2^2 \|\phi\|_{L^2}\|\nabla \phi  \|_{L^2} \|A \psi \|_{L^2}, \]
where $C_1$ and $C_2$ are the sharp constants of the following inequalities:
\begin{align*}
\| f\| _{L^4(\mathbb{R}^4)} \leq C_1 \| \nabla f \|_{L^2(\mathbb{R}^4)}  \qquad \| f\|_{L^{8/3}(\mathbb{R}^4)} \leq C_2 \| f\|_{L^2(\mathbb{R}^4)}^{1/2} \|\nabla f\|_{L^2(\mathbb{R}^4)}^{1/2}.
\end{align*}
By our construction of $\phi$ and the assumption that $\|u_0\|_{L^2} <  \sqrt{2}/(C_1C_2^2)$, we have at each time step $\|\phi\| < (1 + N^{-s})\sqrt{2} /(C_1C_2^2)$. By choosing $N$ large, we also have $\|\phi\|_{L^2} <  \sqrt{2}/(C_1C_2^2)$ at each step. Choose $C_0<1$ so that $\|\phi\|_{L^2} C_1C_2^2 < \sqrt{2} C_0$. Then we have 
\[ E(\phi, \psi)  = \|A \psi\|_{L^2}^2 + 2\|\nabla \phi \|_{L^2}^2 -2 \sqrt{2} C_0 \|\nabla \phi  \|_{L^2} \|A \psi \|_{L^2} \gtrsim  \|A \psi\|_{L^2}^2 + 2\|\nabla u \|_{L^2}^2. \]
Thus $E( \phi, \psi) \approx \|\phi\|_{\dot{H}^1}^2 + \|\psi\|_{H^1}^2$ at each time step.

\section{Proof of Proposition \ref{schrodinger_est}} \label{estimateProof}

By duality, to obtain the smoothing estimate \eqref{eq:sch} it suffices to show that 
\[ \iint uvw \d x \d t = \iint \widehat{uv}(\xi_0, \tau_0) \widecheck{w}(\xi_0, \tau_0) \d \xi_0 \d \tau_0 \lesssim \|u\|_{X^{s,b}} \|v\|_{X^{r,b}_\pm} \|w\|_{X^{-(s + \alpha),1-b}}. \]
We introduce the functions $f_i$, which allow us to state the estimate in terms of $L^2$ norms:
\[ f_1 = \lb \xi \rb^{s} \lb \tau + |\xi|^2 \rb ^{b} \hat{u}, \quad f_2 = \lb \xi \rb^{r} \lb \tau \pm |\xi| \rb ^{b} \hat{v},  \quad \text{and}\quad f_3 = \lb \xi \rb^{-(s + \alpha)} \lb \tau + |\xi|^2 \rb ^{1-b} \hat{w} . \]
Using these functions and the convolution structure of $\widehat{uv}$, the required estimate takes the form
\begin{equation}\label{eq:mainEst}
 \iiiint\limits_{\substack{\sum \xi_i = 0 \\ \sum \tau_i = 0 }} \frac{\lb \xi_0\rb^{s + \alpha} \lb \xi_1\rb^{-s} \lb \xi_2 \rb ^{-r} f_0(\xi_0, \tau_0) f_1(\xi_1,\tau_1)f_2(\xi_2,\tau_2)}{\lb \tau_0 - |\xi_0|^2 \rb^{1-b} \lb \tau_1 + |\xi_1|^2 \rb^b \lb \tau_2 \pm | \xi_2 |  \rb^b} \d\xi_1 \d\xi_2 \d\tau_1 \d\tau_2 \lesssim  \prod_{i = 0}^2 \|f_i\|_{L^2_{\xi,\tau}}. 
\end{equation} 
Before proceeding, we state a one-dimensional calculus lemma which will be used repeatedly. For proofs of similar results, see \cite{ET2}.
\begin{lem}\label{calcEst}
 If $\alpha > 1$ and $\alpha \geq \beta \geq 0 $, then  
\[\int_\mathbb{R} \frac{\d y}{\lb y - a \rb^{\alpha} \lb y-b \rb^{\beta}} \lesssim \lb a - b\rb ^{- \beta}. \] 
\end{lem}
We proceed with the proof by breaking the integration region into many components and considering each separately. \\

\noindent{\underline{\textbf{CASE 0.} $|\xi_1|, |\xi_2| \lesssim 1$}. 
We ignore the order one multipliers $\lb \xi_0\rb^{s + \alpha} \lb \xi_1\rb^{-s} \lb \xi_2 \rb ^{-r}$  on the left-hand side of \eqref{eq:mainEst} and work with 
\begin{align*}
 \iiiint\limits_{\substack{\sum \xi_i = 0 \\ \sum \tau_i = 0 }} &\frac{f_0(\xi_0, \tau_0) f_1(\xi_1,\tau_1)f_2(\xi_2,\tau_2)}{\lb \tau_0 - |\xi_0|^2 \rb^{1-b} \lb \tau_1 + |\xi_1|^2 \rb^b \lb \tau_2 \pm | \xi_2|\rb^b}  \d\xi_1 \d\xi_2 \d\tau_1 \d\tau_2 \\
&\lesssim \|f_0\|_{L^2} \left\| \lb \tau_0 - |\xi_0|^2 \rb^{b-1}\iint\frac{f_1(\xi_1, \tau_1)f_2(-\xi_0-\xi_1, -\tau_0-\tau_1)}{ \lb \tau_1 + |\xi_1|^2 \rb^b \lb-\tau_0 - \tau_1 \pm | \xi_0+\xi_1|\rb^b}   \d\xi_1  \d\tau_1\right\|_{L^2_{\xi_0,\tau_0}}.
\end{align*}

Using Cauchy-Schwartz in $\d \xi_1 \d \tau_1$ and then in $\d \xi_0 \d \tau_0$, the $\xi_0,\tau_0$ norm in the previous line is bounded by 
\begin{align*}
&\left\| \iint \hspace{0pt}f_1^2(\xi_1, \tau_1)f_2^2(-\xi_0-\xi_1, -\tau_0-\tau_1) \d \xi_1 \d \tau_1 \iint \hspace{0pt}\frac{\lb \tau_0 - |\xi_0|^2 \rb^{2b-2} \d\xi_1  \d\tau_1}{ \lb \tau_1 + |\xi_1|^2 \rb^{2b}\lb-\tau_0 - \tau_1 \pm | \xi_0+\xi_1|\rb^{2b}}  \right\|_{L^1_{\xi_0,\tau_0}}^{1/2} \\
&\lesssim \left( \sup\limits_{\xi_0, \tau_0} \quad \iint \frac{\lb \tau_0 - |\xi_0|^2 \rb^{2b-2}}{ \lb \tau_1 + |\xi_1|^2 \rb^{2b} \lb-\tau_0 - \tau_1 \pm | \xi_0+\xi_1|\rb^{2b}}\d\tau_1 \d\xi_1 \right)^{1/2} \\
& \hspace{2in}  \times \left\| \iint f_1^2(\xi_1, \tau_1)f_2^2(-\xi_0-\xi_1, -\tau_0-\tau_1) \d \xi_1 \d \tau_1 \right\|_{L^1_{\xi_0,\tau_0}}^{1/2}.
\end{align*}
Notice that the $L^1$ norm on the last line is  $\bigl( \|f_1^2\|_{L^1_{\xi,\tau}}\|f_2^2\|_{L^1_{\xi,\tau}}\bigr)^{1/2} = \|f_1\|_{L^2_{\xi,\tau}}\|f_2\|_{L^2_{\xi,\tau}}$, so we need only show that the supremum is finite. This is simple when $|\xi_1| \lesssim 1$. First use the fact that $\lb a + b \rb \lesssim \lb a \rb \lb b \rb$ to obtain
\begin{align*} \iint \frac{\lb \tau_0 - |\xi_0|^2 \rb^{2b-2}\lb \tau_1 + |\xi_1|^2 \rb^{-2b}}{  \lb-\tau_0 - \tau_1 \pm | \xi_0+\xi_1|\rb^{2b}}\d\tau_1 \d\xi_1  
\lesssim \iint \frac{\lb \tau_0 + \tau_1 - |\xi_0|^2 + |\xi_1|^2 \rb^{2b-2}}{ \lb-\tau_0 - \tau_1 \pm | \xi_0+\xi_1|\rb^{2b}} \d\tau_1 \d\xi_1. 
\end{align*}
Apply Lemma \ref{calcEst} and the fact that the integral is constrained to the region $|\xi_1| \lesssim 1$ to bound the supremum by 
\begin{align*}
\sup_{\xi_0} &\int \lb  |\xi_1|^2 - |\xi_0|^2 \pm |\xi_0 + \xi_1|  \rb^{2b-2} \d\xi_1   \lesssim 1. 
\end{align*}

This finishes with the region where all the $\xi_i$ are small. The argument holds in any dimension and puts no constraints on $s$, $r$, or $\alpha$. 

For the remaining cases, the resonances of the equation play a significant role. To deal with them, we need a few definitions. Let $\alpha$ denote the angle between $\xi_1$ and $\xi_2$. We define the maximum modulation $M$ as follows, and use the fact that $\tau_0 + \tau_1 + \tau_2 = 0$ and $\xi_0 + \xi_1 + \xi_2 = 0$ to bound it:
\begin{align*}  M = \max\left\{| \tau_0 - |\xi_0|^2|, |\tau_1 + |\xi_1|^2|, |\tau_2 \pm | \xi_2 | | \right\} &\gtrsim \left| |\xi_0|^2 - |\xi_1|^2 \mp |\xi_2|  \right| \\
&= 2 |\xi_1||\xi_2| \left| \cos \alpha + \frac{|\xi_2| \mp 1 }{2|\xi_1|} \right| = 2 |\xi_1 || \xi_2| A. 
\end{align*}
We also need a dyadic decomposition. Let $f_i^{M^j} = f_i\big|_{\{|\xi|\approx M^j\}}$ for $M^j$ dyadic so that $f_i = \sum_{M^j} f_i^{M^j} $. Then from \eqref{eq:mainEst}, it suffices to show that 
\begin{equation}\label{eq:dyadEst}
\sum\limits_{\substack{M_i^j \\ \text{dyadic}}} \;\;  \iiiint\limits_{\substack{\sum \xi_i = 0 \\ \sum \tau_i = 0 }} \frac{f_0^{M_0^j}(\xi_0, \tau_0)f_1^{M_1^j}(\xi_1, \tau_1)f_2^{M_2^j}(\xi_2, \tau_2) \d\xi_1\d\xi_2 \d\tau_1\d\tau_2}{\lb \xi_0\rb^{-(s + \alpha)} \lb \xi_1\rb^{s} \lb \xi_2 \rb ^{r}  \lb \tau_0 - |\xi_0|^2 \rb^{1-b} \lb \tau_1 + |\xi_1|^2 \rb^b \lb \tau_2 \pm | \xi_2 |  \rb^b} \lesssim  \prod_{i = 0}^2 \|f_i\|_{L^2_{\xi,\tau}}. 
\end{equation} 
In the following cases, we drop the $M_i^j$ superscript to lighten the notation, and implicitly assume that $f_i$ is supported on the dyadic shell $|\xi| \approx M_i$. This results in estimates which depend on the $M_i$. To finish the proof, we show in each case that these estimates can be dyadically summed to yield \eqref{eq:dyadEst}. \\

\noindent\underline{ \textbf{CASE 1.} $M = |\tau_0 - |\xi_0|^2|$}. In this case, we must control 
\begin{equation} \label{eq:case1}
 \lb M_0\rb^{s + \alpha} \lb M_1 \rb^{-s} \lb M_2 \rb^{-r} \iiiint\limits_{\substack{\sum \xi_i = 0 \\ \sum \tau_i = 0 }} \frac{f_0(\xi_0, \tau_0) f_1(\xi_1,\tau_1)f_2(\xi_2,\tau_2)}{\lb M \rb ^{1-b}\lb \tau_1 + |\xi_1|^2 \rb^b \lb \tau_2 \pm |\xi_2|\rb^b}   \d\xi_1\d\xi_2 \d\tau_1\d\tau_2. 
 \end{equation}
Recall that $M \gtrsim M_1M_2 |A|$. We first consider the nonresonant regions, i.e. where $|A| \gtrsim 1$. \\

\noindent\underline{\textbf{Case 1.1.} $M_0 \lesssim M_1 \approx M_2$ and $|A| \gtrsim 1$}. Estimate the maximum $M$ by $M_1M_2$. Decompose the functions $f_1$ and $f_2$ parabolically: 
\begin{align*}
f_1 &= \sum_{n\in \Z} f_1^n, \quad \text{ where }\quad f_1^n = f_1\chi_{\{\tau_1 + |\xi_1|^2 = n + \mathcal{O}(1)\}}\\
f_2 &= \sum_{m \in \Z} f_2^m, \quad \text{ where }\quad  f_2^m = f_2\chi_{\{\tau_2 \pm | \xi_2|  = m + \mathcal{O}(1)\}}.
\end{align*}
Using this decomposition, controlling the integral in \eqref{eq:case1} amounts to bounding
\begin{align} \label{eq:int1}
\begin{split}
\sum_{n \in \Z} \sum_{m  \in \Z} \lb n \rb ^{-b} \lb m \rb^{-b} \iiiint\limits_{\substack{ \theta_i = \mathcal{O}(1) }} f_0&(-\xi_1-\xi_2, |\xi_1|^2 \pm | \xi_2 | -n -m - \theta_1 - \theta_2)\\
\times \; &f_1^n(\xi_1, -|\xi_1|^2 + n + \theta_1) f_2^m(\xi_2, \mp | \xi_2 | + m + \theta_2) \d\xi_1\d\xi_2 \d\theta_1 \d\theta_2.
\end{split}
\end{align}

To do so, recall we're assuming $M_0 \lesssim M_1 \approx M_2$, and decompose the supports of $f_1$ and $f_2$ into squares (or in higher dimensions, hypercubes) of side length $L \approx  M_0$. Denote these squares by $\{ |\xi_1| \approx M_1\} = \bigcup_i  Q_i$ and $\{|\xi_2| \approx M_2\} = \bigcup_j  R_j $. Note that since $\sum \xi_i = 0$ and $M_0 \lesssim M_1, M_2$, the square $R_j = R_{j(i)}$ is essentially determined by the square $Q_i$. Technically, each region $Q_i$ could correspond to up to $3^d$ of the $R_j$ regions, but this factor does not harm the estimates. Let $f^n_{1,Q_i} = f^n_1\chi_{\{\xi_1 \in Q_i\}}$ and $f^m_{2,R_j} = f^m_2 \chi_{\{\xi_2 \in R_j\}}$. \\

For the moment, consider only the inner $\d \xi_1 \d\xi_2$ integral in \eqref{eq:int1}. For fixed $\theta_i$, $n$, and $m$, change variables by letting $u = -\xi_1 - \xi_2$ and $v = |\xi_1|^2 \pm | \xi_2 | - n -m -\theta_1 - \theta_2$. We will use $u$ and $v$ to replace $\xi_1$ and one component of $\xi_2$. Let $\xi_i = (\xi_{i,1}, \xi_{i,2}, \ldots, \xi_{i,d})$. Computing the Jacobian matrix for the change of variables in the two-dimensional case gives 
\begin{equation*}
 \begin{bmatrix} 
  \frac{\d u_1}{\d\xi_{1,1}}       & \frac{ \d u_1}{\d\xi_{1,2}}       & \frac{ \d u_1}{\d\xi_{2,1}}       & \frac{\d u_1}{\d\xi_{2,2}} \\ 
  \frac{\d u_2}{\d\xi_{1,1}}       & \frac{ \d u_2}{\d\xi_{1,2}}       & \frac{ \d u_2}{\d\xi_{2,1}}       & \frac{\d u_2}{\d\xi_{2,2}} \\
  \frac{\d v}{\d\xi_{1,1}}         & \frac{ \d v}{\d\xi_{1,2}}         & \frac{ \d v}{\d\xi_{2,1}}         & \frac{\d v}{\d\xi_{2,2}}   \\
  \frac{\d \xi_{2,2}}{\d\xi_{1,1}} & \frac{ \d \xi_{2,2}}{\d\xi_{1,2}} & \frac{ \d \xi_{2,2}}{\d\xi_{2,1}} & \frac{\d \xi_{2,2}}{\d\xi_{2,2}} 
 \end{bmatrix} = 
 \begin{bmatrix} 
  -1       &  0       & -1       &  0 \\ 
   0       & -1       &  0       & -1 \\
   2\xi_{1,1} & 2\xi_{1,2} & \pm \frac{\xi_{2,1}}{|\xi_2|} & \pm \frac{\xi_{2,2}}{|\xi_2|} \\ 
   0          & 0          & 0                             & 1 
 \end{bmatrix}
\end{equation*}
when we replace $\xi_1$ and $\xi_{2,1}$ by $u$ and $v$. The result in the case when we replace $\xi_{2,2}$ instead of $\xi_{2,1}$ is similar -- the one in the final row just moves a column to the left. Computing the determinant of the Jacobian matrix, we see that 
\begin{align*}
\d u \d v  \d\xi_{2,2} = \left| 2\xi_{1,1} \mp \frac{\xi_{2,1}}{|\xi_2 |} \right| \d\xi_1 \d\xi_2 \quad \text{ and } \quad 
\d u  \d v \d \xi_{2,1} = \left| 2\xi_{1,2} \mp \frac{\xi_{2,2}}{| \xi_2 |} \right| \d\xi_1 \d\xi_2,
\end{align*}
depending on which component of $\xi_2$ we retain. In higher dimensions the result is parallel:
\[ \d u \d v \d \xi_{2,1} \cdots \d \xi_{2,j-1} \d \xi_{2, j+1} \cdots \d \xi_{2,d} = \left| 2 \xi_{1,j} \pm \frac{\xi_{2,j}}{|\xi_2|}\right| \d \xi_1 \d \xi_2. \] 

We may assume that $M_1 \gg 1$; Case 0 dealt with the region where all $M_i$ are small. Then we have $|\xi_{1,j}| \approx M_1 \gg 1$ for some $j$. Without loss of generality, assume that $|\xi_{1,1}| \approx M_1$. Let $\pi: \R^d \to \R^{d-1}$ be the projection onto the last $d-1$ components. Define
\[H(u,v,\xi_{2,2}, \xi_{2,3}, \ldots, \xi_{2,d}) = f^n_{1,Q_i}(\xi_1, -|\xi_1|^2 +n+ \theta_1)f^m_{2,R_{j(i)}}(\xi_2, \mp | \xi_2 | +m+ \theta_2).\]
Then the $\d \xi_1 \d\xi_2$ integral in \eqref{eq:int1} is bounded by 
\begin{align*} 
&\sum_{Q_i}\; \iiint\limits_{(\xi_{2,2}, \ldots, \xi_{2,d}) \in \pi(R_{j(i)})} f_0(u,v) H(u,v,\xi_{2,2}, \ldots, \xi_{2,d}) \left| 2\xi_{1,1} \mp \frac{\xi_{2,1}}{| \xi_2|} \right| ^{-1}  \d u  \d v  \d\xi_{2,2} \ldots \d \xi_{2,d}  \\
\lesssim&\; M_1^{-1} \| f_0\|_{L^2}\sum_{Q_i}  \left\| \int_{(\xi_{2,2}, \ldots, \xi_{2,d}) \in \pi(R_{j(i)})}  H(u,v,\xi_{2,2}, \ldots, \xi_{2,d}) \d\xi_{2,2} \ldots \d \xi_{2,d} \right\|_{L^2_{u,v}} \\ 
\leq&\; L^{(d-1)/2} M_1^{-1} \| f_0\|_{L^2}\sum_{Q_i}  \left\| H(u,v,\xi_{2,2}, \ldots, \xi_{2,d}) \right\|_{L^2_{u,v,\xi_{2,2}, \ldots, \xi_{2,d}}}\\ 
\lesssim&\; L^{(d-1)/2} M_1^{-1/2} \| f_0\|_{L^2_{\xi, \tau}}\sum_{Q_i} \left\| H(\xi_1,\xi_2) \right\|_{L^2_{\xi_1,\xi_2}} \\ 
=&\; L^{(d-1)/2} M_1^{-1/2} \| f_0\|_{L^2_{\xi, \tau}} \sum_{Q_i}\| f^n_{1,Q_i}(\xi_1, -|\xi_1|^2 +n+ \theta_1)\|_{L^2_{\xi_1}} \|f^m_{2,R_{j(i)}}(\xi_2, \mp | \xi_2 | +m+ \theta_2)\|_{L^2_{\xi_2}} \\
\leq&\;  L^{(d-1)/2} M_1^{-1/2} \| f_0\|_{L^2_{\xi, \tau}} \| f^n_{1}(\xi_1, -|\xi_1|^2 +n+ \theta_1)\|_{L^2_{\xi_1}} \|f^m_{2}(\xi_2, \mp | \xi_2 | +m+ \theta_2)\|_{L^2_{\xi_2}}.
\end{align*}
The last inequality follows from applying Cauchy-Schwarz to the $Q_i$ sum. Thus \eqref{eq:int1} is bounded by 
\begin{align*}
&L^{(d-1)/2} M_1^{-1/2} \| f_0\|_{L^2_{\xi, \tau}}
\sum_{n \in \Z} \sum_{m  \in \Z} \lb n \rb ^{-b} \lb m \rb^{-b} \iint\limits_{\theta_i = \mathcal{O}(1) }
 \| f^n_{1}(\xi_1, -|\xi_1|^2 +n+ \theta_1)\|_{L^2_{\xi_1}} \\
 & \hspace{3in} \times \|f^m_{2}(\xi_2, \mp | \xi_2 | +m+ \theta_2)\|_{L^2_{\xi_2}} \d \theta_i 
 \end{align*}
By Cauchy-Schwarz in $\theta_1$ and $\theta_2$, using the fact that $\theta_i  = \mathcal{O}(1)$, and then in $n$ and $m$ using the fact that $b > \frac12$, bound this by
 \begin{align*}
 & L^{(d-1)/2} M_1^{-1/2} \| f_0\|_{L^2_{\xi, \tau}}
\sum_{n \in \Z} \lb n \rb ^{-b}  \| f^n_{1}(\xi_1, -|\xi_1|^2 +n+ \theta_1)\|_{L^2_{\xi_1, \theta_1}(\theta_1 = \mathcal{O}(1))} \\
&\hspace{2in}  \times \sum_{m  \in \Z}\lb  m \rb^{-b} \|f^m_{2}(\xi_2, \mp | \xi_2 | +m+ \theta_2)\|_{L^2_{\xi_2, \theta_2}(\theta_2 = \mathcal{O}(1))} \\
& \hspace{0.5in}\lesssim  L^{(d-1)/2} M_1^{-1/2} \prod_{i=0}^2 \| f_i\|_{L^2_{\xi, \tau}}
\end{align*}

So in this case, the left-hand side of \eqref{eq:case1} is bounded by 
\[ \lb M_0\rb^{s + \alpha} \lb M_1 \rb^{-s-r-2(1-b) - \frac12} L^{\frac{d-1}2} \approx \lb M_0\rb^{s + \alpha}M_0^{\frac{d-1}2} \lb M_1 \rb^{-s-r-2(1-b) - \frac12}.\] 
This is dyadically summable if  $ \alpha < r + 2 -\frac{d}2 $ for $b - \frac12>0 $ sufficiently small. \\

\noindent\underline{\textbf{Case 1.2.} $M_1 \ll M_0 \approx M_2$ and $|A| \gtrsim 1$}. 
In this case, we have $M_0 \approx M_2 \gg 1$. Thus $M \gtrsim ||\xi_0|^2 - |\xi_1|^2 \mp |\xi_2 | | \approx M_0^2$. 

If $M_1 \gg 1$, we proceed just as in the previous case. (The restriction $M_1 \gg 1$ is necessary to ensure that the Jacobian is nonzero.) Break the shells $\{ \xi_0 : |\xi_0| \approx M_0\}$ and $\{ \xi_2 : |\xi_2 | \approx M_2\}$ into squares (or hypercubes) of side length $L \approx M_1$ and change variables as in Case 1.1.  This results in the bound 
\[ \lb M_0 \rb^{s + \alpha - r - 2(1-b)} \lb M_1\rb^{-s} M_1^{-\frac12} L^{\frac{d-1}2} \approx \lb M_0 \rb^{s + \alpha - r - 2(1-b)} \lb M_1\rb^{-s + \frac{d-2}2}\]
for the left-hand side of \eqref{eq:case1}, which is dyadically summable $s + \alpha < r + 1 + \min\{ 0, s - \frac{d-2}{2}\}$, i.e. when $\alpha < \min\{ r -s + 1, r - \frac{d-4}2\}$, for $b - \frac12$ sufficiently small. 

When $M_1 \lesssim 1$, we can use an argument similar to that in Case 0. Notice that 
\begin{align*}
 \iiiint\limits_{\substack{\sum \xi_i = 0 \\ \sum \tau_i = 0 }} &\frac{f_0(\xi_0, \tau_0) f_1(\xi_1,\tau_1)f_2(\xi_2,\tau_2)}{ \lb \tau_1 + |\xi_1|^2 \rb^b \lb \tau_2 \pm | \xi_2|\rb^b}  \; d\xi_1\d\xi_2 d\tau_1\d\tau_2 \\
&\lesssim \|f_0\|_{L^2_{\xi,\tau}} \left\| \iint\frac{f_1(\xi_1, \tau_1)f_2(-\xi_0-\xi_1, -\tau_0-\tau_1)}{ \lb \tau_1 + |\xi_1|^2 \rb^b \lb-\tau_0 - \tau_1 \pm | \xi_0+\xi_1|\rb^b}  \; d\xi_1\; d\tau_1\right\|_{L^2_{\xi_0,\tau_0}(|\xi_0| \approx M_0)}.
\end{align*}
Using Cauchy-Schwartz in $\d \xi_1 \d \tau_1$ and then in $\d \xi_0 \d \tau_0$, the last norm above is bounded by 
\begin{align*}
&\left\| \iint f_1^2(\xi_1, \tau_1)f_2^2(-\xi_0-\xi_1, -\tau_0-\tau_1) \d \xi_1 \d \tau_1 \iint\limits_{|\xi_1| \approx M_1} \frac{ \lb \tau_1 + |\xi_1|^2 \rb^{-2b}}{ \lb-\tau_0 - \tau_1 \pm | \xi_0+\xi_1|\rb^{2b}}  \; d\xi_1\; d\tau_1\right\|_{L^1_{\xi_0,\tau_0}(|\xi_0| \approx M_0)}^{1/2} \\
&\lesssim \left( \sup\limits_{\substack{|\xi_0| \approx M_0 \\ \tau_0}} \quad \iint\limits_{|\xi_1| \approx M_1}  \frac{ \lb \tau_1 + |\xi_1|^2 \rb^{-2b}}{ \lb-\tau_0 - \tau_1 \pm | \xi_0+\xi_1|\rb^{2b}}\d\tau_1 \d\xi_1 \right)^{1/2} 
   \left\| f_1^2(-\cdot , - \cdot) * f_2^2(\cdot,\cdot) \right\|_{L^1_{\xi_0,\tau_0}}^{1/2} \\
&\lesssim  M_1^{d/2}\|f_1^2\|_{L^1}^{1/2}\|f_2^2\|_{L^1}^{1/2} =  M_1^{d/2}\|f_1\|_{L^2}\|f_2\|_{L^2}. 
\end{align*}
The rough bound on the supremum is obtained as follows. Using Lemma \ref{calcEst} a) gives
\begin{align*}  \sup_{\substack{|\xi_0| \approx M_0\\ \tau_0}} \quad \iint\limits_{|\xi_1| \approx M_1} \lb \tau_1 + |\xi_1|^2\rb^{-2b} \lb -\tau_0 -\tau_1 \pm | \xi_0 + \xi_1 | \rb ^{-2b}\d\tau_1 \d\xi_1  \\
\lesssim \sup_{\xi_0, \tau_0} \quad \iint\limits_{|\xi_1| \approx M_1} \lb | \xi_1|^2 \pm | \xi_0 + \xi_1| - \tau_0 \rb ^{-2b} \d \xi_1 \lesssim M_1^d.
\end{align*}
In this case, we thus bound the left-hand side of \eqref{eq:case1} by 
\[ \lb M_0\rb^{s + \alpha - r - 2(1-b)}\lb M_1 \rb^{-s} M_1^\frac{d}2, \]
which is summable when $s + \alpha < r + 1$ for $M_1 \lesssim 1$ as long as $b-\frac12$ is sufficiently small. \\

\noindent \underline{\textbf{Case 1.3.} $M_2 \ll M_0 \approx M_1$ and $|A| \gtrsim 1$}.
The same procedure as in Case 1.1 works here, with the simplification that no decomposition of the integration regions into squares is required. (The decomposition served to ensure that the projection on the integration region in $\xi_2$ onto any axis had measure at most $\min\{M_0, M_1, M_2\}$, which is automatically true when $M_2 = \min\{M_0, M_1, M_2\}$). Repeating the change of variables and ensuing argument gives 
\begin{align*}
\iiiint\limits_{\substack{\sum \xi_i = 0 \\ \sum \tau_i = 0 }} \frac{f_0(\xi_0, \tau_0) f_1(\xi_1,\tau_1)f_2(\xi_2,\tau_2)}{ \lb \tau_1 + |\xi_1|^2 \rb^b \lb \tau_2 \pm \lb \xi_2\rb\rb^b}  \d\xi_1\d\xi_2 \d\tau_1\d\tau_2 
\lesssim M_2^{\frac{d-1}2} M_0^{-\frac12} \prod_{i=0}^2 \|f_i\|_{L^2_{x,t}}. 
\end{align*}
Thus the left-hand size of \eqref{eq:case1} can be estimated by 
\[ \lb M_0 \rb^{ \alpha - (1 - b)-\frac12} \lb M_2\rb^{-r} M_2^{\frac{d-1}2 -(1-b)}.\]
For $M_2 \lesssim 1$, this sums as long as $ \alpha < 1$.  When $M_2 \gg 1$, the product is summable when $\alpha < 1 + \min\{0, r - \frac{d-2}2\}$, i.e. when $\alpha < \min\{1 , r - \frac{d-4}{4}\}$. \\

\noindent\underline{\textbf{Case 1.4. Resonance.} $|A| \ll 1$}. 
Recall that $A =  \cos \alpha + \frac{|\xi_2| \mp 1 }{2|\xi_1|} $, so when $A$ is small, $|\xi_2| \lesssim |\xi_1|$. Thus we assume that $|\xi_1| \gg 1$, since otherwise all $|\xi_i| \lesssim 1$. That region was addressed in Case 0. 

Decompose parabolically as in Case 1.1 and take another dyadic decomposition about the resonant surface: assume $|A| \approx \nu \ll 1$ dyadic. Then we need to control
\begin{equation}\label{eq:case1.4}
\begin{split}
 &\sum_{\nu \ll 1}\frac{1}{\nu^{1-b}}\sum_{n \in \Z} \sum_{m  \in \Z} \lb n \rb ^{-b} \lb m \rb^{-b} \iiiint\limits_{\substack{ \theta_i =  \mathcal{O}(1) \\ |A| \approx \nu}} f_0(-\xi_1-\xi_2, |\xi_1|^2 \pm | \xi_2 | -n -m -\theta_1 - \theta_2) \\
 &\times f^n_1(\xi_1, -|\xi_1|^2+n + \theta_1) f^m_2(\xi_2, \mp | \xi_2 | +m+ \theta_2) \d\xi_1\d\xi_2 \d \theta_1 \d \theta_2. 
\end{split}
\end{equation}
To visualize the region of integration, consider a fixed $\xi_1$. The resonant surface $A = 0$ in  $\xi_2$-space is then a slightly distorted version of a hypersphere of radius $|\xi_1|$ centered at $-\xi_1$. This sphere has equation $|\xi_2|^2 + 2 |\xi_1||\xi_2| \cos \alpha = 0$, while the actual resonant surface satisfies $|\xi_2|^2 + 2 |\xi_1||\xi_2| \cos \alpha \mp | \xi_2 |= 0$. The region of integration in $\xi_2$ is a shell centered on this curve, with thickness $\lesssim \nu M_1$. This holds since for a fixed $\xi_1$ and a fixed angle $\alpha$, 
\[ A \in [\nu, 2 \nu ] \quad \Rightarrow \quad |\xi_2| \in \Big[2|\xi_1|(\nu - \cos \alpha) \pm 1, 2|\xi_1|(2\nu - \cos \alpha) \pm 1\Big],\]
an interval of length $2 \nu |\xi_1|$. See Figure \ref{fig} for a plot of this region in $\R^2$ for $\xi_1 \in \mathbb{R}^+$. 

\begin{figure}
 \centering 
 \caption{Integration region in $\xi_2$} 
 \includegraphics[width=0.5\textwidth]{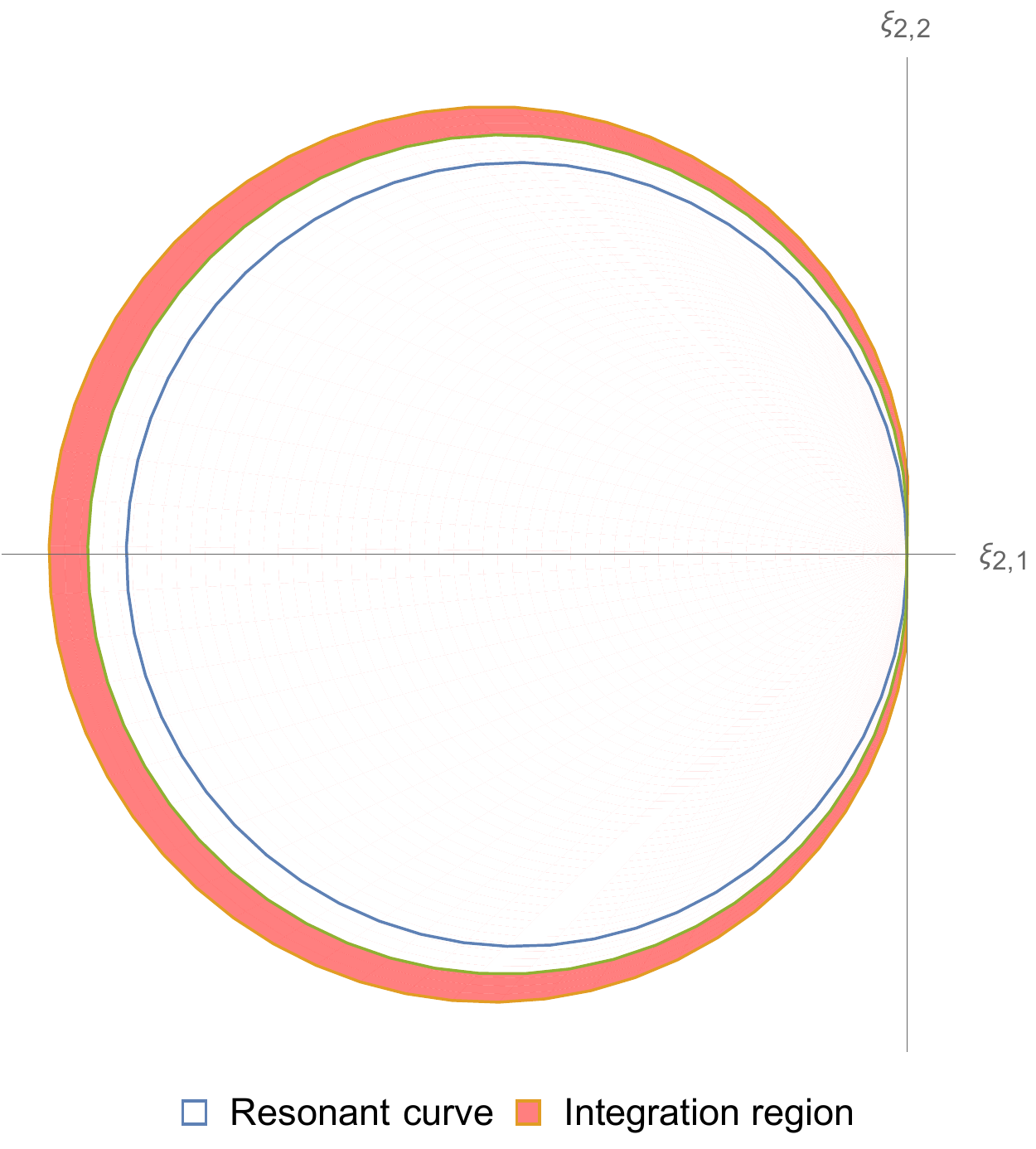}
 \label{fig}
\end{figure}

Decompose the annulus $\{ |\xi_1| \approx M_1\}$ into two parts --  a set $B$ where $| \xi_{1,i} | \approx M_1$ for each $i$, and its complement. In two dimensions, this decomposition can be described explicitly by taking
\begin{align*}
B = \left\{ \xi_1 \; : \;  |\xi_1| \approx M_1,\;  \arg(\xi_1) \in \left[\frac{\pi}{8}, \frac{3\pi}{8}\right) \cup \left[\frac{5\pi}{8}, \frac{7\pi}{8} \right)\cup \left[\frac{9\pi}{8}, \frac{11\pi}{8}\right) \cup \left[\frac{13\pi}{8}, \frac{15\pi}{8}\right) \right\}.
\end{align*}
Notice that the complement of $B$ is simply a rotation of $B$ about the origin. In higher dimensions, the set B is similar -- if we describe the space in hyperspherical coordinates, we require all $d-1$ angular variables to be bounded away from multiples of $\pi/2$ -- specifically to fall in the intervals $[\tfrac{n\pi}{8}, \tfrac{(n+2)\pi}{8})$ given above. The complement of $B$ then consists of $2^{d-1} -1$ copies of $B$, each of which can be obtained from $B$ by a sequence of $\pi/4$ radian rotations.

The remainder of this calculation will consider the two-dimensional case. There is no fundamental difference in higher dimensions; only much more onerous notation. We perform a rotation so that \eqref{eq:case1.4} can be written as a sum of two integrals over $B$. In the following $R_y$ denotes a rotation by $y$ radians. 
\begin{equation*}\label{eq:case1.4.2}
\begin{split}
 \sum_{k=0}^1 \sum_{\nu \ll 1}\frac{1}{\nu^{1-b}}\sum_{n \in \Z} \sum_{m  \in \Z} \lb n \rb ^{-b} \lb m \rb^{-b} \iiiint\limits_{\substack{ \theta_i =  \mathcal{O}(1) \\ |A| \approx \nu \\ \xi_1 \in B}} f_0(R_{\frac{k\pi}{4}}(-\xi_1-\xi_2), |\xi_1|^2 \pm | \xi_2 | -n -m -\theta_1 - \theta_2) \\
\hspace{0.5in} \times f^n_1(R_{\frac{k\pi}{4}}(\xi_1), -|\xi_1|^2+n + \theta_1) f^m_2(R_{\frac{n\pi}{4}}(\xi_2), \mp | \xi_2 | +m+ \theta_2) \d\xi_1\d\xi_2\d\theta_1\d\theta_2. 
\end{split}
\end{equation*}

Now break the $\d \xi_1 \d \xi_2$ integration into two additional cases: one where for fixed $\xi_1$ and $\xi_{2,1}$, the projection of the integration region onto the $\xi_{2,2}$ axis is length $\lesssim \nu M_1$, and one where for fixed $\xi_1$ and $\xi_{2,2}$, the projection onto the $\xi_{2,1}$ axis is length $\lesssim \nu M_1$. Once again use the change of variables from Case 1.1: set $u = -\xi_1-\xi_2$ and $v = |\xi_1|^2 \pm | \xi_2 | -n -m -\theta_1 - \theta_2$. In the first region, when the projection onto the $\xi_{2,2}$ axis is small, change variables to replace $ \d \xi_1 \d \xi_2$ with $\d \xi_{2,2} \d u \d v$. When the projection onto the $\xi_{2,1}$ axis is small, use $\d \xi_{2,1} \d u \d v$. 

Following exactly the same steps as in Case 1.1, we bound the \eqref{eq:case1.4} by 
\[ \sum_{\nu \ll 1} \frac{1}{\nu^{1-b}}(\nu M_1)^{\frac12}/M_1^{\frac12} \lesssim 1, \]
using the fact that $b > \frac12$. In general dimensions, the bound is $M_2^{\frac{d-2}2}$. Thus the quantity to be dyadically summed is 
\[ \lb M_0 \rb^{s + \alpha} \lb M_1\rb^{-s} \lb M_2 \rb^{-r} M_1^{b-1} M_2^{b-1}M_2^{\frac{d-2}2}. \]
When $A$ is small, there are only two possibilities: either $M_0 \approx M_1 \approx M_2 \gg 1$ or $M_2 \ll M_0 \approx M_1$. In the first case, we must sum $\lb M_0\rb^{ \alpha -r - 2(1-b)+ \frac{d-2}2}$, which is possible when $\alpha < r + 2 - \frac{d}2 = r - \frac{d-4}{2}$. 

In the second case, when $M_2 \gtrsim 1$ we get $\lb M_0 \rb^{\alpha - (1-b)} \lb M_2 \rb^{-r + \frac{d-2}2 - (1-b)}$, which sums as long as $ \alpha < \frac12 + \min\{ 0 , r  - \frac{d-3}{2}\} = \min\{\frac12, r - \frac{d-4}2\}$ for $b-\frac12$ sufficiently small. 
When $M_2 \ll 1$, we estimate the maximum modulation multiplier by one instead of $AM_1M_2$, and use the argument in Case 1.3 (merely drop the $M_i^{-(1-b)}$ factors) to get convergence when $\alpha < \frac12$. \\

\noindent \underline{\textbf{CASE 2.} $M = |\tau_1 + |\xi_1|^2|$}. \\

\noindent \underline{\textbf{Case 2.1.} $M_0 \ll M_1 \approx M_2$ and $|A| \gtrsim 1$}. 
When $M_0 \lesssim 1$,  the supremum argument in Case 1.2 applies. It gives the multiplier $\lb M_0\rb^{s + \alpha} \lb M_1\rb^{-s-r - 2(1-b)}M_0^\frac{d}2$, which sums when $s + r >- 1$, a condition which is always met when $s,r > - \frac12$. When $M_0 \gg 1$, decompose the $M_1$ and $M_2$ annuli into squares of scale $M_0$ and use a change of variables just as in Case 1.1 to get the multiplier $\lb M_0\rb^{s + \alpha + \frac{d-2}2} \lb M_1\rb^{-s-r - 2(1-b)}$. This sums when $\max\{s + \alpha + \frac{d-2}2,0\} < s + r + 1$, which holds when $s, r > -\frac12$ and $\alpha < r - \frac{d-4}{2}$. \\

\noindent \underline{\textbf{Case 2.2.} $M_1 \lesssim M_0 \approx M_2$ and $|A| \gtrsim 1$}.
Apply the argument in Case 1.1 to obtain the multiplier $\lb M_0\rb^{s + \alpha -r - 2(1-b) - \frac12} \lb M_1\rb^{-s}M_1 ^\frac{d-1}2$. When $M_1 \lesssim 1$, this sums if $\alpha < r -s  + \frac{3}2$. If $M_1 \gg 1$, we need $s + \alpha < \min\{s - \frac{d-1}2 ,0\}+r + \frac{3}2$, i.e. $\alpha < \min\{ r-s + \frac32, r - \frac{d-4}{2}\}$.\\

\noindent \underline{\textbf{Case 2.3.} $M_2 \ll M_0 \approx M_1$ and $|A| \gtrsim 1$}.
Proceed as in Case 1.3 to obtain the multiplier 
\[\lb M_0\rb^{ \alpha - (1-b) - \frac12} \lb M_2 \rb^{-r} M_2^{\frac{d-1}2-(1-b)}.\]
When $M_2 \lesssim 1$, this sums as long as $ \alpha <  1$. When $M_2 \gg 1$, we require $ \alpha  < 1 + \min\{r - \frac{d-2}2, 0 \}= \min\{1, r - \frac{d-4}{2}\}$.\\

\noindent \underline{\textbf{Case 2.4. Resonance.} $|A| \ll 1$}. 
The procedure is the same as that in Case 1.4, merely exchanging the roles of $(\xi_1, \tau_1)$ and $(\xi_2, \tau_2)$, and yields the same constraints.\\

\noindent \underline{\textbf{CASE 3.} $M = |\tau_2 \pm \lb \xi_2 \rb|$}. 
Here we must control 
\[ \lb M_0\rb^{s + \alpha} \lb M_1 \rb^{-s} \lb M_2 \rb^{-r}M_1^{-(1-b)}M_2^{-(1-b)} \iiiint\limits_{\substack{\sum \xi_i = 0 \\ \sum \tau_i = 0 }} \frac{f_0(\xi_0, \tau_0) f_1(\xi_1,\tau_1)f_2(\xi_2,\tau_2)}{|A|^{1-b} \lb \tau_0 - |\xi_0|^2 \rb^{b} \lb \tau_1 + |\xi_1|^2 \rb^b}  \d\xi_1\d\xi_2\d\tau_1\d\tau_2. \]
For the 2d case, the estimates in \cite{CDKS} give 
\[  \iiiint\limits_{\substack{\sum \xi_i = 0 \\ \sum \tau_i = 0 }} \frac{f_0(\xi_0, \tau_0) f_1(\xi_1,\tau_1)f_2(\xi_2,\tau_2)}{ \lb \tau_0 - |\xi_0|^2 \rb^{b} \lb \tau_1 + |\xi_1|^2 \rb^b}  \d\xi_1\d\xi_2\d\tau_1\d\tau_2 \lesssim \left( \frac{\min\{M_0,M_1\}}{\max\{M_0,M_1\}}\right)^{1/2} \prod_{i=0}^2 \|f_i\|_{L^2_{x,t}}.  \]
However, for some cases the argument relies on the $L^4L^4$ Strichartz estimate, which does not hold for $d \neq 2$. \\

\noindent \underline{\textbf{Case 3.1}. $M_0 \lesssim M_1 \approx M_2$ and $|A| \gtrsim 1$}.
In this case, the arguments from \cite{CDKS} can be applied directly to give a multiplier of $M_0^{\frac{d-1}2}M_1^{-\frac12}$. This means that we must dyadically sum 
\[\lb M_0\rb^{s + \alpha} M_0^{\frac{d-1}2} \lb M_1 \rb^{-s-r - 2(1-b) - \frac12}.\]
When $M_0 \lesssim 1$, this sums as long as $s + r > - \frac32$, a condition which is certainly met for $s,r > -\frac12$. When $M_0 \gg 1$, we need  $ \alpha  < r - \frac{d-4}2$. \\
 
\noindent \underline{\textbf{Case 3.2.} $M_1 \ll M_0 \approx M_2$ and $|A| \gtrsim 1$}.
Here again the results from \cite{CDKS} can be applied. Doing so yields $\lb M_0 \rb^{s + \alpha -r - 2(1-b) - \frac12} \lb M_1 \rb^{-s} M_1^{\frac{d-1}2}$. When $M_1 \lesssim 1$, we require $\alpha < r  -s + \frac32$. When $M_1 \gg 1$, we need $s + \alpha - r -2 (1-b) - \frac12 + \max\{\frac{d-1}2 - s ,0 \} < 0$, which holds when $\alpha < r-s + \frac32$ and $\alpha < r - \frac{d-4}{2}$.  \\

\noindent \underline{\textbf{Case 3.3.} $M_2 \ll M_0 \approx M_1$ and $|A| \gtrsim 1$}.
When $M_2 \lesssim 1$, use the supremum argument which appears in Case 1.2. This yields $\lb M_0 \rb^{s + \alpha - s - (1-b) } \lb M_2 \rb^{-r} M_2^{-(1-b)}M_2^{\frac{d-1}2}$. This sums when $\alpha < \frac12$ for $b - \frac12 > 0$ sufficiently small.
When $M_2 \gg 1$, decompose the $M_0$ and $M_1$ annuli into squares of scale $M_2$ and change variables. Unlike the previous cases though, the change of variables here gives a Jacobian of order $M_2$ (see \cite{CDKS} for details). Thus we arrive at $\lb M_0 \rb^{s + \alpha - s - (1-b) } \lb M_2 \rb^{-r + \frac{d-2}2-(1-b)}$. To sum this dyadically, we need $ \alpha  < \frac12 + \min\{0, r -\frac{d-3}2\} = \min\{ \frac12, r- \frac{d-4}{2}\}$. \\

\noindent \underline{\textbf{Case 3.4. Resonance.} $|A| \ll 1$}.
When $1 \ll M_2 \lesssim M_0 \approx M_1$, proceed as in Case 1.4. Dyadically decompose $f_0$ and $f_1$ and then change variables by letting $u = -\xi_0 - \xi_1$ and $v = | \xi_1|^2 - |\xi_0|^2 - n -m - \theta_1 - \theta_2$. This leads to a Jacobian of 
\[ \d u \d v \d \xi_{0,1} = |\xi_{0,2} + \xi_{1,2}| \d \xi_0 \d\xi_1 \quad  \text{ or } \quad \d u \d v \d \xi_{0,2} = |\xi_{0,1} + \xi_{1,1}| \d \xi_0 \d\xi_1. \]
The result in higher dimensions is similar:
\[ \d u \d v \d \xi_{0,1} \cdots \d \xi_{0,j-1} \d \xi_{0, j+1} \cdots \d \xi_{0, d} = |\xi_{2,j}| \d \xi_0 \d \xi_1. \]
Proceed just as in Case 1.4 to arrive at $\lb M_0 \rb^{ \alpha- (1-b)} \lb M_2 \rb^{-r - (1-b) + \frac{d-2}2}$, which sums when $ \alpha  < \frac12 + \min\{0, r - \frac{d-3}2\} = \min\{ \frac12, r - \frac{d-4}{2}\}$. 

When $M_2 \lesssim 1$, estimate the maximum modulation by $1$ instead of by $M_1^{-(1-b)}M_2^{-(1-b)}A$ and apply the argument from Case 1.3 (again merely removing the $M_i^{-(1-b)}$ factors) to conclude summability when $ \alpha < \frac{1}2$. 


\section*{Acknowledgments}
The author would like to thank Nikolaos Tzirakis for his guidance and for many helpful discussions. 


\end{document}